\documentclass[reqno]{amsart}
\usepackage[all,ps,cmtip]{xy}
\usepackage{amssymb}
\usepackage{mathrsfs}
\usepackage{enumerate}
\usepackage{enumitem}
\usepackage{xcolor}

\DeclareMathOperator*{\colim}{\mathrm{colim}}
\newcommand{\mrm}[1]{\mathrm{#1}}
\newcommand{\mc}[1]{\mathscr{#1}}
\newcommand{\wt}[1]{\widetilde{#1}}

\newcommand{\Tor}{\mc{T}or}
\newcommand{\Et}{\mathrm{Et}}
\newcommand{\Ord}{\mc{O}rd}
\newcommand{\Hom}{\mrm{Hom}}

\newcommand{\Ext}{\mc{E}xt}

\newcommand{\F}{\mathbf{F}}
\newcommand{\id}{\mrm{id}}

\newcommand{\sHom}{\mc{H}om}
\newcommand{\mto}{\mapsto}

\newcommand{\Aff}{\mrm{Aff}}
\newcommand{\isomto}{\stackrel{\sim}{\longrightarrow}}

\newcommand{\Z}{\mathbf{Z}}

\newcommand{\Sh}{\mathrm{Sh}}
\newcommand{\Fr}{\mathrm{Fr}}
\newcommand{\alg}{\mathrm{alg}}

\newcommand{\Spf}{\mathrm{Spf}}
\newcommand{\mul}{\mathrm{mult}}
\newcommand{\et}{\mathrm{\acute{e}t}}

\newcommand{\sA}{\mathscr{A}}
\newcommand{\sO}{\mathscr{O}}

\let\: \colon
\newcommand{\Set}{\mathrm{Set}}
\newcommand{\op}{\mathrm{op}}

\DeclareMathOperator{\Spec}{\mathrm{Spec}}
\newcommand{\etsheaf}{\Sh} 
\newcommand{\bZ}{\mathbf{Z}}

\newcommand{\bF}{\mathbf{F}}

\newcommand{\bN}{\mathbf{N}}

\def\longisomap{\,{\buildrel \sim\over\longrightarrow}\,}
\newcommand{\jp}{u}
\newcommand{\trunc}{\tau}
\newcommand{\dt}{\delta}
\newcommand{\dtSh}[1]{\Sh^{\dt}_{#1}}

\newcommand{\dtoind}{\mathscr{C}}

\newcommand{\dtind}[1]{(\Aff_{#1,\mathrm{ind}})^{\dt}}

\newcommand{\displaylabelfork}[6]{{	\entrymodifiers={+!!<0pt,\fontdimen22\textfont2>}
	\def\objectstyle{\displaystyle}
\xymatrix{{#1} \ar^-{#2}[r] & {#3} \ar@<0.7ex>^-{#4}[r]\ar@<-0.7ex>_-{#5}[r] & {#6}}}}

\newcommand{\longlabelmap}[1]{{\,\buildrel #1\over\longrightarrow\,}}
\newcommand{\longmap}{{\longlabelmap{}}}
\newcommand{\ghost}{w}
\newcommand{\coghost}{\gamma}

\begin{document}
\author{James Borger}
\address{Mathematical Sciences Institute, Australian National University, Canberra, Australia}
\email{james.borger@anu.edu.au}

\author{Lance Gurney}

\title{Canonical lifts and $\dt$-structures}
\newtheorem{thmx}{Theorem}
\renewcommand{\thethmx}{\Alph{thmx}} 
\newtheorem{theo}[subsubsection]{Theorem}
\newtheorem*{theo*}{Theorem}
\newtheorem{coro}[subsubsection]{Corollary}
\newtheorem{lemm}[subsubsection]{Lemma}
\newtheorem{prop}[subsubsection]{Proposition}
\newtheorem{exmp}[subsubsection]{Example}
\theoremstyle{remark}
\newtheorem{rema}[subsubsection]{Remark}

\maketitle
\setcounter{tocdepth}{1}

\begin{abstract}
We extend the Serre--Tate theory of canonical lifts of ordinary abelian varieties to arbitrary unpolarised families of ordinary abelian varieties parameterised by a $p$-adic formal scheme $S$.
We show that the canonical lift is the unique lift to $W(S)$ which admits a $\delta$-structure
in the sense of Joyal, Buium, and Bousfield. We prove analogous statements
for families of ordinary $p$-groups and $p$-divisible groups.
\end{abstract}
\tableofcontents

\section*{Introduction}
Let $p$ be a prime number. Serre and Tate proved that any ordinary abelian variety $A$ over a perfect field $k$
of characteristic $p$ admits a distinguished lift to an abelian scheme $\wt{A}$ over $W(k)$, the ring of Witt
vectors with entries in $k$. It is in fact the unique lift, up to unique isomorphism, which admits an
endomorphism $\varphi$ reducing to the Frobenius morphism modulo $p$ and lying over the Frobenius endomorphism
of $W(k)$. The original reference for this is Messing's book~\cite{Messing72},  p.\ 177.

The purpose of this paper is to extend this to (unpolarised) ordinary abelian schemes over any $p$-adic base $S$:

\begin{thmx} Let $A$ be an ordinary abelian scheme over $S$.
	Then $A$ has a unique lift to an abelian scheme $\wt{A}$ over
	$W(S)$ admitting a $\dt_{W(S)}$-structure compatible with its group structure.
\end{thmx}

The meaning of the terms in this theorem will be explained precisely in the body of the paper. But we can say a
few words now informally. The $\dt$-structures referred to are those of Joyal, Buium, and Bousfield,
generalised to our setting. To a first approximation, a $\dt$-structure is equivalent to a lift of the Frobenius
endomorphism, but $\dt$-structures are better in that they repair certain problems that Frobenius lifts have in
the presence of $p$-torsion. (They also happen to be a $p$-typical analogue of the $\lambda$-ring structures of
algebraic K-theory and have other names that reflect this, such as $p$-typical $\lambda$-structures,
$\Lambda_p$-structures, and $\theta^p$-structures, but we will not make use of this.) A $\dt_{W(S)}$-structure 
is simply a $\dt$-structure compatible with the canonical $\dt$-structure on $W(S)$

The term \emph{$p$-adic} above means that the map $S\to\Spec(\bZ)$ factors through $\Spf(\bZ_p)\to\Spec(\bZ)$.
So if $S$ is affine, it means that $p$ is nilpotent on $S$. In general, we allow $S$ to be any sheaf of sets on
the category of affine schemes with respect to the \'etale topology. In particular, $S$ can be any $p$-adic
formal scheme. Note however that the passage from $S$ affine to $S$ arbitrary is purely formal. So the reader
can take $S$ to be affine with no real loss in generality. But as usual, allowing $S$ to be arbitrary from the 
outset allow us some flexibility at no cost.

Finally, our Witt vector construction $W(S)$ is a certain ind-object $\colim_n W_n(S)$. For example, if
$S=\Spec(R)$, then $W(S)$ is the ind-scheme $\colim_n \Spec(W_n(R))$. In that case, an abelian scheme $A$ over
$W(S)$ can be identified with a compatible family of abelian schemes $A_n$ over the truncated Witt vector rings
$W_n(R)$, as $n$ varies, and a $\dt_{W(S)}$-structure on $A$ can be expressed entirely in terms of the
compatible family of $A_n$ instead of the limiting ind-scheme $A$. In the special case where $R$ is a perfect
$\bF_p$-algebra, $W(\Spec(R))$ can be identified with the usual $p$-adic formal scheme $\Spf(W(R))$ and then
$\wt{A}$ can be viewed as a $p$-adic formal scheme over $\Spf(W(R))$. But in general $W(\Spec(R))$ and $\wt{A}$
will not be $p$-adic formal schemes. 

A consequence of the theorem above is the following one:
\begin{thmx}
The category of ordinary abelian schemes over $S$ is equivalent to the category of ordinary
abelian schemes over $W(S)$ equipped with a $\dt_{W(S)}$-structure compatible with the group structure.
\end{thmx}
This confirms the philosophy (explicitly stated in~\cite{Borger:LRFOE} and implicit in much of Buium's work
going back to~\cite{Buium:Diff-chars-over-p-adic}) that a $\dt_{W(S)}$-structure can be viewed as descent data
for the non-existent map $W(S)\to S$. 

The proof goes by proving an analogous statement, theorem~\ref{theo:ordinary-lambda}, for ordinary $p$-groups.
The corresponding statement for $p$-divisible groups (remark~\ref{rmk:ind-groups}) is a formal consequence of
this, and then we prove the theorem above by applying the theorem of Serre and Tate equating deformations
of an abelian variety with deformations of its $p$-divisible group.

The study of canonical lifts in families began with the papers of Finotti~\cite{Finotti:Mazur-Tate}, answering a
question of Mazur and Tate, and Erdo\u{g}an~\cite{Erdogan:canonical-lifts}. In our earlier
paper~\cite{Borger-Gurney:Nagoya}, we observed that in the case of elliptic curves, the canonical lift
construction extends rather formally to arbitrary families. It is worth recalling the approach we took there in
order to contrast it to the one of the present paper.

First, given a $\delta$-structure on a ring $A$, the universal property of Witt vectors discovered by Joyal
gives a bijection between maps $\Spec(R)\to \Spec(A)$ and $\delta$-equivariant maps $\Spec(W(R))\to\Spec(A)$.
Viewing $\Spec(A)$ as a moduli space $X$, one could say that a $\delta$-structure on $X$ provides a theory of
canonical lifts for the objects parametrised by $X$. (One could reasonably take this as a definition of a
theory of canonical lifts.) Second, if $X$ is flat over $\bZ$ locally at $p$, then giving a $\delta$-structure
on $X$ is equivalent to giving a lift $\varphi\:X\to X$ of the Frobenius map. Putting the two together, we see
that on a moduli space which is flat over $\bZ$ at $p$, a Frobenius lift on the moduli space formally gives
rise to a theory of canonical lifts in arbitrary families.

This generalises to non-affine moduli spaces $X$, once one extends the definitions. It even works without any
real complications when $X$ is a stack, still assumed flat at $p$, if it admits an \'etale cover by affine
schemes equipped with lifts of the Frobenius, at least given the \'etale descent theorems quoted in
section~\ref{subsec:w-etale} below. In particular, if $X$ is the moduli stack of polarised ordinary abelian
$p$-adic formal schemes, then every such abelian scheme $A$ admits a canonical subgroup $C_A\subset A$ lifting the the kernel of the Frobenius endomorphism (cf.\ Adreatta--Gasbarri \cite{andreatta_gasbarri_2007} where more general canonical subgroups are constructed without the ordinary hypothesis, generalising earlier work of Katz \cite{Katz} in the case of elliptic curves). The functor \[X\to X: A\mto A/C_A,\] which sends an ordinary polarised abelian
scheme to the quotient by the canonical subgroup is then a lift of the Frobenius, and so the flatness of $X$ implies
the existence of a theory of canonical lifts. This is the approach taken in~\cite{Borger-Gurney:Nagoya} for the
moduli of elliptic curves. So if we were only interested in canonical lifts of families of \textit{polarised}
abelian varieties, as in section~\ref{sec:polar}, we could argue as in~\cite{Borger-Gurney:Nagoya} and dispense
with much of this paper.

But this no longer holds for unpolarised families of abelian varieties because the moduli stack $X$ is not
algebraic. While a $\delta$-structure on $X$ does give a theory of canonical lifts (once $\delta$-structures on
stacks are defined), it is not clear how to upgrade a Frobenius lift on a non-algebraic stack to a
$\delta$-structure. One could attempt a descent from the moduli stack
$X^{\mathrm{pol}}$ of polarised abelian varieties to the moduli stack $X$ of unpolarised abelian varieties.
But as far as we know, the map $X^{\mathrm{pol}}\to X$ does not have any good descent-theoretic properties which
would allow for a shorter argument than the one we have given here. Indeed, it is certainly far from being anything like flat in the case of abelian schemes of dimension $g>1$: the tangent spaces of $X^{\mathrm{pol}}$ have dimension $\frac{1}{2}g(g+1)$ while those of $X$ have dimension  $g^2$.
 
So we construct the $\delta$-structure on $X$ by hand. That said, we prefer to come full circle and use the
more direct language of canonical lifts instead of the more abstract but equivalent language of
$\delta$-structures on the moduli stack. In these terms, the canonical lift $\wt{A}/W(S)$ is singled out as
simply the unique lift admitting a $\delta$-structure. Note that this property can be stated without having to
refer to the moduli space, or to any families at all besides the given one. This is in contrast to the approach
of~\cite{Borger-Gurney:Nagoya} with elliptic curves. There the uniqueness property is only proved for the
functor $A\mapsto \wt{A}$, where $A$ and $S$ must be allowed to variable.

There are essentially two technical points in the proof of our theorem. The first is an analysis of the
deformation theory of ordinary $p$-divisible groups $G/S$ along the map $S\to W(S)$, which allows us to conclude
the existence of canonical lifts of ordinary abelian schemes from the usual Serre--Tate theorem. The second
involves showing that these canonically lifted abelian schemes, which come with a lift of the Frobenius by
construction, actually have a unique $\delta_{W(S)}$-structure.

It is perhaps worth mentioning other possible approaches to the first point above and the problems that
arise when one attempts to use them. First, $p$-divisible groups are often understood in terms of associated
linear algebraic gadgets: Dieudonn\'e crystals in de Jong--Messing \cite{de1999crystalline} or Lau
\cite{lau2018divided}, or slightly different ones such as the displays and so on used by Zink
\cite{zink2001windows}, \cite{zink2002display} and Lau \cite{lau2008displays}. Indeed, such gadgets are much simpler to work with and
it is often substantially easier to construct canonical lifts of the linear algebraic incarnation of ordinary
objects. However, the relationship between $p$-divisible groups and the linear algebraic gadgets mentioned
above is never perfect. In order to obtain an equivalence of categories, and hence have anything usable, one
must place restrictions either on the base, e.g.\ $\F_p$-algebras with bounded nilpotents properties, or on the
$p$-divisible groups, e.g.\ formal groups only.

A second approach might be to consider extending the canonical coordinates on the groups of extensions of
multiplicative $p$-divisible groups by \'etale $p$-divisible groups as in \cite{Katz}. Indeed, if $S$ is a
scheme on which $p$ is nilpotent and $G_\et$ and $G_\mul$ are \'etale and multiplicative $p$-divisible groups
over $S$, then there is a canonical map \begin{equation}\Hom_S(T_p(G_\et), G_\mul)\to \mathrm{Ext}_S(G_\et,
G_\mul).\label{eqn:can-coord}\end{equation} It is easy to construct canonical lifts of $G_\et$ and $G_\mul$ to
$p$-divisible groups schemes over $W(S)$ equipped with $\delta_{W(S)}$-structures. This combined with the
functoriality of (\ref{eqn:can-coord}) shows that any extension in its image can be `canonically lifted' to a
$p$-divisible group over $W(S)$. However, (\ref{eqn:can-coord}) is often not surjective as its image consists
only of extensions which are split on $S_\mathrm{red}$. Of course, one might try to work locally and descend to
this case, but as the Witt vector functor is only compatible with \'etale covers and not more general ones
(e.g.\ flat covers) this approach never really gets off the ground.

One final possible approach is the direct, deformation theoretic method following Nori--Srinivas in the
appendix to \cite{mehta1987varieties}. However, when one works over general rings $A$, e.g.\ any ring which is
not a perfect $\bF_p$-algebra, the Witt vectors are substantially less well behaved. For instance, the
kernels of the projection maps $W_{n+1}(A)\to W_n(A)$ are no longer generated by powers of $p$, nor are they
even principal, which leaves several of the elementary observations of \cite{mehta1987varieties} without any
obvious or useful generalisations. These problems aside, the main lifting result of \cite{mehta1987varieties}
relies on projectivity assumptions which we do not make here.

In the final section of the paper, we show that the canonical lift functor is compatible with duality and polarisations. From this it follows that, in the polarised situation, the canonical lift can be algebraised:
\begin{thmx}\label{thm:delta-algebraisation}
	Let $S=\Spec(R)$, where $R$ is a ring in which $p$ is nilpotent, and let $A/S$ be an ordinary polarised abelian 
	scheme. Then $\wt{A}$ algebraises to a unique polarised abelian scheme $\wt{A}^{\mathrm{alg}}$ over 
	$\Spec(W(R))$. It is the unique lifting of $A/S$ to $\Spec(W(R))$ equipped with a
	$\dt_{\Spec(W(R))}$-structure compatible with its group structure.
\end{thmx}
This reduces to Messing's characterisation of the canonical lift cited above when $R$ is a perfect field,
because for $p$-torsion free schemes every Frobenius lift comes from a unique $\delta$-structure. As mentioned
above, the second assertion in the theorem could no doubt be proved directly without developing the general,
unpolarised theory of the rest of the paper. Also note that the affineness assumption on $S$ is not so much a
restriction but an assumption needed to make sense of the question. Without it, it would not be clear what should
play the role of the completed Witt vector construction $\Spec(W(R))$.

\section{$\dt$-structures and Witt vectors}
The primary purpose of this section is to define $\dt$-structures on abelian schemes over Witt vector
ind-schemes, or more generally on any sheaf of sets. This is done in~\ref{sec:rel-lambda}.
Before that, we recall the basic notions we will need in the theory of Witt vectors and $\dt$-rings from
\cite{Joyal:Witt},\cite{Buium:Diff-chars-over-p-adic},\cite{Bousfield:loop-spaces},\cite{Borger:BGWV-I},\cite{Borger:BGWV-II}. 

The generality in which we work, especially Witt vectors of sheaves of sets, might strike some readers as
excessive. The reason we do it is mainly to give ourselves a convenient, well-behaved category in which we can
perform all limit constructions without any worries. In principle it would be routine to replace all limit
constructions by compatible systems of schemes. But we have found that the bookkeeping needed for this is a
bigger mental burden than broadening our scope a bit and using sheaf theory. And at least
the sheaf theory keeps the paper shorter. This is discussed in more detail in the first paragraphs of \S2
of~\cite{Borger-Gurney:Nagoya}.

\subsection{$\dt$-rings} 
\label{sec:lambda-structure}
A $\dt$-ring is a pair $(R, \delta)$ where $R$ is a ring (commutative)
and $\delta\: R\to R$ is a map satisfying the following identities:

\begin{enumerate}
\item $\delta(x+y)=\delta(x)+\delta(y)-\sum_{i=1}^{p-1}\frac{1}{p}\binom{p}{i} x^i y^{p-i}$
\item $\delta(xy)=x^p\delta(y) + \delta(x)y^p + p\delta(x)\delta(y)$
\item $\delta(1)=0$
\end{enumerate}
A homomorphism of $\dt$-rings $(R, \delta)\to (R', \delta')$ is a homomorphism $\alpha\: R\to R'$ such that 
$\delta'\circ \alpha=\alpha\circ \delta$. This structure was introduced by Joyal~\cite{Joyal:Witt} in
his alternative approach to the Witt vectors, which we follow here, and
later by Buium~\cite{Buium:Diff-chars-over-p-adic} and Bousfield~\cite{Bousfield:loop-spaces}, 
independently and for different purposes.

If $(R, \delta)$ is a $\dt$-ring, then the map $\varphi\: R\to R$ defined by 
\begin{equation}
	\label{eq:phi-definition}
\varphi(x)=x^p+p\delta(x)	
\end{equation}
is a Frobenius lift, i.e., a ring homomorphism which reduces modulo $p$ to the $p$-th power map. 
If $R$ is $p$-torsion free, then the two structures determine each other:
given any Frobenius lift $\varphi'\:R \to R$, there exists a unique $\dt$-ring structure on $R$
such that $\varphi=\varphi'$; and a ring homomorphism between two $p$-torsion free $\dt$-rings 
is a morphism of $\dt$-rings if and only if it commutes with the Frobenius lifts.

But in the presence of $p$-torsion, a Frobenius lift is not a well-behaved structure---for instance, the category
of rings with Frobenius lift
does not have equalisers. The $\dt$-ring structure repairs such flaws and gives an intelligent
generalisation of the structure of a Frobenius lift to all rings. Some evidence for this is the fact that the
forgetful functor $(R, \delta)\mto R$ from $\dt$-rings to rings admits both a right adjoint 
\begin{equation}
	R\mto (W(R), \delta_{W(R)})
\end{equation}
and a left adjoint 
\begin{equation}
	\label{eq:circle-product-def}
	R\mto (J(R), \delta_{J(R)}).	
\end{equation}
In fact, more is true. This forgetful functor is comonadic, meaning that the category of $\dt$-rings agrees
with the category of coalgebras for $W$, viewed as a comonad on the category of rings. It is also monadic:
$\dt$-rings are algebras for the monad $J$.

\subsection{Witt vectors}
The ring $W(R)$ is in fact the usual ring of $p$-typical Witt vectors with entries in $R$, and the corresponding
Frobenius lift $\varphi$ is the Witt vector Frobenius, more commonly denoted $F$. For formal reasons, $W(R)$ has
the following description, which is not the traditional one. As sets, we have 
$W(R)=R\times R\times R\times\cdots$ and $\delta(x_0,x_1,\dots)=(x_1,x_2,\dots)$. The ring structure is given by
laws
\begin{align*}
	(x_0,x_1,\dots) + (y_0,y_1,\dots) &= (S_0,S_1,\dots) \\
	(x_0,x_1,\dots) \cdot (y_0,y_1,\dots) &= (P_0,P_1,\dots)
\end{align*}
where $S_n$ is the polynomial $S_n(x_0,y_0,\dots,x_n,y_n)$ with integer coefficients
giving the Leibniz rule for the operation $\delta^{\circ n}$ with respect to addition, in the sense that
$$
\delta^{\circ n}(x+y)=S_n\big(x,y,\delta(x),\delta(y),\dots,\delta^{\circ n}(x),\delta^{\circ n}(y)\big),
$$
and $P_n$ is similarly the polynomial giving the Leibniz rule for $\delta^{\circ n}$
with respect to multiplication. The existence of such
polynomials can be seen by applying the basic Leibniz rules (1)--(3) above repeatedly, and the uniqueness
can be seen by reducing to the case of $p$-torsion-free rings, where a $\dt$-structure is equivalent to a
Frobenius lift.
The additive and multiplicative neutral elements are $(0,0,0,\dots)$ and $(1,0,0,\dots)$.

There is a canonical isomorphism between $W(R)$ as defined above and Witt's original
construction~\cite{Witt:Vectors}, as recalled for example in Serre's book~\cite{Serre:LocalFields}. This
follows from the fact that the original construction satisfies
the universal property of ours. For a proof of this, one can see p.\ 215 of Lazard's 
book~\cite{Lazard:formal-groups-book}, keeping in mind
that it is enough to restrict to rings $R$ which are $p$-torsion free since both functors are represented by
rings which are $p$-torsion free---in fact by polynomial rings. 

Beware however that the canonical isomorphism between our construction and Witt's is not the identity! It is
the identity on the components $x_0$ and $x_1$ but not on $x_2$ and the higher components. One could say that
there are two different coordinate systems on the same
functor---the ones above, which we call the \emph{Buium--Joyal components}, and the traditional ones, which we
call the \emph{Witt components}. The Buium--Joyal components are directly connected to a simple universal
property, as above, and hence are usually better for conceptual purposes.
For example, the comonad structure map (or coplethysm or Artin--Hasse map)
	$$
	\Delta\:W(R)\to W(W(R))
	$$
does not have a simple explicit description in terms of the Witt components, but in terms of 
the Buium--Joyal components it does:
	$$
	\Delta\:(x_0,x_1,\dots) \mapsto ((x_{0},x_{1},\dots),(x_{1},x_{2},\dots),(x_{2},x_{3},\dots),\dots).
	$$
On the other hand, the Witt components are closer to
the Verschiebung operator and are sometimes more convenient for computations.
\subsection{Truncations}
It also follows from the construction of $W(R)$ above that for any integer $n\geq 0$, the quotient
	$$
	W_n(R)=R^{n+1}
	$$
of $W(R)$ consisting of truncated vectors $(x_0,\dots,x_n)$ is a quotient ring. Indeed,
the Leibniz rules for $\delta^{\circ n}$ depend only on the operators $\delta^{\circ i}$ for $i\leq n$.
Then $W(R)$ is naturally identified with the limit of the projective system of rings
	$$
	\cdots \longmap W_n(R) \longlabelmap{\trunc} \cdots \longlabelmap{\trunc} W_1(R) \longlabelmap{\trunc} W_0(R)
	$$
given by the truncation maps $\trunc\:(x_0,\dots,x_n)\mapsto (x_0,\dots,x_{n-1})$.

For $n \geq 0$ the Verschiebung ideal $V^{n+1}W(R)$ is defined to be the kernel of the truncation map:
	$$
	0 \to V^{n+1}W(R) \to W(R) \to W_{n}(R) \to 0.
	$$
It can also be expressed as the image of the $n$-th iterate of the usual Verschiebung 
operator $V\:W(R)\to W(R)$, but since it is not effortless to define $V$ in terms of the Buium--Joyal components 
and since we will not need it, we can ignore this.

The operator $\delta\:W(R)\to W(R)$ descends to the truncations but only at the expense of a shift in degree:
we have a set map $\delta\:W_n(R)\to W_{n-1}(R)$ given by 
	$$
	\delta(x_0,\dots,x_n)=(x_1,\dots,x_n).
	$$ 
Similarly, the Frobenius map $\varphi\:W(R)\to W(R)$ descends to a ring homomorphism
$W_n(R)\to W_{n-1}(R)$ given by $\varphi(x)=\trunc(x)^p+p\delta(x)$.

The comonad structure map also descends to the truncations in a degree-shifting sense. It becomes
a family of maps 
\begin{equation}
	\label{eq-truncated-comonad}
	\Delta\:W_{m+n}(R)\to W_n(W_m(R))
\end{equation}
which send a truncated Witt vector $(x_0,\dots,x_{m+n})$ to
	\begin{equation}
	\label{map:coplethysm}
	\big((x_0,x_1,\dots,x_m),(x_1,x_2,\dots,x_{m+1}),\dots,(x_n,x_{n+1},\dots,x_{m+n})\big).
	\end{equation}
	
\subsection{Ghost and coghost maps}
Given a $\dt$-ring $(R,\delta)$, we have an associated pair $(R,\varphi)$ of a ring with an endomorphism,
where $\varphi$ is the Frobenius lift defined in~(\ref{eq:phi-definition}).
This defines a functor from the category of $\dt$-rings to the category of rings with an endomorphism.

The forgetful functor from the category of rings with endomorphism to rings also has a right adjoint 
$R\mapsto (\Pi(R),\varphi_{\Pi(R)})$, where $\Pi(R)$ is $R\times R\times \cdots$ with, now, the product ring
structure and where $\varphi_{\Pi(R)}$ is defined by shifting:
$$\varphi_{\Pi(R)}\:\langle z_0,z_1,\dots\rangle\mapsto \langle z_1,z_2,\dots\rangle.$$
The induced map of rings with endomorphism
	$$
	\ghost\:W(R) \to \Pi(R)
	$$
is the so-called ghost map. In coordinates, it sends a Witt vector $(x_0,x_1,\dots)$ 
in Buium--Joyal components to 
$\langle Z_0,Z_1,\dots\rangle$, where $Z_n$ is the integral polynomial $Z_n(x_0,\dots,x_n)$ which expresses
$\varphi^{\circ n}$ in terms of the iterates of $\delta$:
	$$
	\varphi^{\circ n} = Z_n(\id,\delta,\delta^{\circ 2},\dots,\delta^{\circ n}).
	$$
Once again, such polynomials exist because we can apply the Leibniz rules (1)--(3) repeatedly, and one
can show they are unique by considering the case of $p$-torsion-free rings. (The polynomials $Z_n$ have no 
simple closed-form description. This is one drawback of the Buium--Joyal components,
compared to the Witt components, where the analogous polynomials are the usual Witt polynomials.)
Finally, the ghost map descends to a homomorphism on the truncations:
	$$
	\ghost_n\:W_n(R) \to R^{n+1},
	$$
where the target has the product ring structure.

\subsection{Witt vectors and jet spaces of sheaves}
\label{subsec:w-etale}
The functor $W_n$ on rings defines a functor on the category $\Aff$
of affine schemes, which we also denote by $W_n$:
\begin{align*}
	\Aff &\longlabelmap{W_n}\Aff \\
	\Spec(B) &\mapsto \Spec(W_n(B)).	
\end{align*}
To pass to non-affine schemes, we need to know this functor is well behaved with respect to localisation.
This is provided by the following theorem:

\begin{theo}\label{thm:w-etale}
	If a morphism $R\to R'$ of rings is \'etale, then so is the induced map $W_n(R)\to W_n(R')$. If
	in addition $R\to S$ is an arbitrary ring map, then the canonical map 
	$W_n(R')\otimes_{W_n(R)} W_n(S)\to W_n(R'\otimes_R S)$ is an isomorphism.
\end{theo}

For a proof, see 9.2 and 9.4 of \cite{Borger:BGWV-I}. 
We will only need this for rings in which $p$ is nilpotent, in which case it
was proved earlier in the appendix of Langer--Zink \cite{Langer-Zink:dRW}. (For the so-called big Witt vectors,
it was proved even earlier in van der Kallen~\cite{van-der-Kallen:Descent}.)

Let us now equip $\Aff$ with the \'etale topology (SGA4~\cite{SGA4.2}, exp.\ VII) and let $\etsheaf$ denote the corresponding category of sheaves of sets. Given any sheaf $X\in\etsheaf$, define the $n$-th arithmetic jet space (or Greenberg transform) 
$J^n(X)\:\Aff^\op\to\Set$ by
	$$
	J^n(X)\: T \mto X(W_n(T)).
	$$
This is in fact a sheaf:

\begin{theo}
	The presheaf $J^n(X)$ on $\Aff$ is a sheaf in the \'etale topology.
	The resulting functor $J^n\:\etsheaf\to\etsheaf$ has
	a left adjoint $W_n\:\etsheaf\to\etsheaf$. The functor $W_n$
	agrees with the usual Witt vector functor on the category of affine schemes:
		$$
		W_n(\Spec(R)) \longisomap \Spec(W_n(R))
		$$
\end{theo}
This follows by general sheaf theory from theorem \ref{thm:w-etale}.

\begin{theo}
	If $X$ is a scheme (or an algebraic space), then so are $J^n(X)$ and $W_n(X)$.
\end{theo}
\begin{proof}
	See 12.1, 15.1, 15.6 of \cite{Borger:BGWV-II}.
\end{proof}

We pass to $n=\infty$, by taking a limit: 
	$$
	J(X)(T) := \lim_n \big(J^n(X)(T)\big) = \lim_n X(W_n(T))
	$$
So for a ring $C$, a $C$-point of $J(X)$ is a compatible family of $W_n(C)$-points of $X$.
In the affine case, $X=\Spec(R)$, we have
	$$
	J(\Spec(R))(\Spec(C)) = \lim_n \Hom(R,W_n(C)) = \Hom(R,W(C)) = \Hom(J(R), C),
	$$
where $J(R)$ denotes the free $\dt$-ring on $R$, as in (\ref{eq:circle-product-def}).
Thus we have a canonical identification
	$$
	J(\Spec(R)) = \Spec(J(R)),
	$$
and so the functor $J$ we just defined on sheaves agrees with the one already defined on rings.

The truncation maps $\trunc\:W_{n+1}(R)\to W_{n}(R)$ on rings induce functorial projections 
\begin{equation}
	\label{map:truncation-definition}
	\jp\:J^{n+1}(X)\to J^n(X),	
\end{equation}
for $X\in\Sh$. They in turn induce
functorial maps $W_n(X)\to W_{n+1}(X)$ for $X\in\etsheaf$, and hence an inductive system
	$$
	W_0(X) \to W_1(X) \to \cdots.
	$$
We then define 
	$$
	W(X)=\colim_n W_n(X).
	$$

Observe that it does not matter whether we take this colimit
in the category of sheaves or presheaves, since every object of $\Aff$ is quasi-compact and quasi-separated
in the \'etale topology. Also, we emphasise that generally one has
	$$
	W(\Spec(R))\neq \Spec(W(R)),
	$$ 
unlike the case for $W_n$ with $n$ finite.
For example, if $R=\bZ/p\bZ$, this becomes the familiar fact $\Spf(\bZ_p)\neq\Spec(\bZ_p)$.

Similarly, the maps $\varphi\: W_n(R)\to W_{n-1}(R)$ induce maps $\varphi\:J^n(X)\to J^{n-1}(X)$ whose limit  
\begin{equation}
	\label{map:phi-J-def}
\varphi_{J(X)}\: J(X)\to J(X)	
\end{equation}
under the projections $\jp$ is a lift of the Frobenius. We also obtain a lift of the Frobenius on the Witt vectors \[\varphi_{W(X)}\: W(X)\to W(X).\]

The ghost maps $\ghost_n\:W_n(C)\to C^{n+1}$ induce maps on sheaves \[\ghost_n\: \coprod_{n+1} X\to W_{n}(X)\] which we also call the ghost maps. In the colimit, they become a map
	\[
	w\: \coprod_{\bN} X\to W(X).
	\] 
The map $X\to W(X)$ obtained by restricting $w$ along the inclusion $X\to \coprod_{\bN} X$ corresponding to $0\in \bN$ is called the first ghost component. 

Finally, we also have \emph{coghost} maps which are dual to the ghost maps:
	$$
	\coghost_{n}\: J^n(X)\to X^{n+1},
	$$
On $C$-points, $\coghost_n$ is defined to be the map
	$$
	J^n(X)(C) = X(W_n(C)) \longlabelmap{X(\ghost_n)} X(C^{n+1}) = X^{n+1}(C).
	$$
The coghost maps are compatible with the projections $u$ and hence pass to a map
	$$
	\coghost\: J(X)\to X^{\bN}
	$$
in the limit.

\begin{theo}[cf.\ 15.2, 15.3 of \cite{Borger:BGWV-II}]\label{theo:witt-et-fibre} Let $Y=\colim_i Y_i$ be an 
ind-algebraic space and let $X\to Y$ be a morphism which is representable (by algebraic spaces) and \'etale. Then for all $m, n \leq \infty$, the following hold:

\begin{enumerate}[label=\textup{(\roman*)}]
\item The induced map $W_n(X)\to W_n(Y)$ is representable and \'etale.
\item The induced map 
	\[
	W_n(X)\to W_{n+m}(X)\times_{W_{n+m}(Y)}W_n(Y)
	\] 
is an isomorphism.
\item For all morphisms $Y'\to Y$ where $Y'$ is an algebraic space the natural map 
	\[
	W_n(X\times_Y Y')\to W_n(X)\times_{W_n(Y)}W_n(Y')
	\] 
is an isomorphism.
\end{enumerate}
\end{theo}
\begin{proof} 
First observe that the case where $m,n<\infty$ implies the general case. Indeed, (ii) and (iii) follow from the finite case by taking colimits; then (i) then follows from (ii) and (iii) and the finite case. So it is enough to assume $m,n<\infty$.

Second observe that if $X, Y, Y'$ are algebraic spaces, this the theorem is 15.2 and 15.3 of 
\cite{Borger:BGWV-II}. 

Now write $X_i=X\times_Y Y_i$. Then $X_i$ is an algebraic space because $X$ is representable over $Y$. 

(iii) If $Y'$ is quasi-compact and quasi-separated, then $Y'\to Y$ factors through $Y_i\to Y$ for some $i$.
So we have 
$$
W_n(X\times_Y Y')=W_n(X_i\times_{Y_i} Y') = W_n(X_i)\times_{W_n(Y_i)} W_n(Y'),
$$
since the result holds for algebraic spaces. Hence we
obtain the commutative diagram \[\xymatrix{W_n(X\times_Y Y') \ar[r] & W_n(X)\times_{W_n(Y)}W_n(Y')\\
W_n(X_i)\times_{W_n(Y_i)}W_n(Y') \ar@{=}[u]\ar[r]& (W_n(X)\times_{W_n(Y)}W_n(Y_i))\times_{W_n(Y_i)}W_n(Y').\ar@{=}[u]}\] Therefore, it is enough to show that \[W_n(X_i)\to W_n(X)\times_{W_n(Y)}W_n(Y_i)\] is an isomorphism. However, this is nothing but the colimit of the maps \[W_n(X_i)\to W_n(X_j)\times_{W_n(Y_j)}W_n(Y_i)\] for $j\geq i$, which are isomorphisms by 15.2 of \cite{Borger:BGWV-II}.

For general $Y'_\lambda$, we can reduce to the case where $Y'$ is quasi-compact and quasi-separated by taking an affine cover $Y'_\lambda\to Y'.$ 
For we then have 
\begin{eqnarray*} 
	W_n(X\times_Y Y')\times_{W_n(Y')}W_n(Y'_\lambda) & = & W_n(X\times_ Y Y'_\lambda)\\ 
	&\isomto& W_n(X)\times_{W_n(Y)}W_n(Y'_\lambda)\\
	&=&(W_n(X)\times_{W_n(Y)}W_n(Y'))\times_{W_n(Y')}W_n(Y'_\lambda)
\end{eqnarray*}
and so the map 
\[W_n(X\times_Y Y')\longmap W_n(X)\times_{W_n(Y)} W_n(Y')\]
is an isomorphism.

(ii) The map in question is the filtered colimit over $i$ of the maps 
\[
W_n(X_i) \to W_{n+m}(X_i)\times_{W_{n+m}(Y_i)}W_n(Y_i)
\] 
which are isomorphisms by 15.4 of \cite{Borger:BGWV-II}.

(i) The family $(W_n(Y_i)\to W_n(Y))_{i}$ is a cover in $\Sh$ and a morphism being representable by algebraic spaces and \'etale is local in the \'etale topology, so that it is enough to show the claim after base change along each $W_n(Y_i)\to W_n(Y)$. 
The map \[W_n(X)\times_{W_n(Y)}W_n(Y_i)\to W_n(Y_i)\] under the isomorphism \[W_n(X)\times_{W_n(Y)}W_n(Y_i)\isomto W(X_i)\] of (iii) is just the natural map $W_n(X_i)\to W_n(Y_i)$ which is \'etale by 15.4 of \cite{Borger:BGWV-II}.
\end{proof}

\subsection{$\dt$-structures on sheaves}\label{subsec:delta-sheaves}
The monad structure on the ring functor $R\mapsto J(R)$ of~(\ref{eq:circle-product-def}) prolongs naturally to
a comonad structure on the sheaf functor $J\:\Sh\to\Sh$, as we now explain.
The truncated comonad maps of
(\ref{eq-truncated-comonad}) induce maps between iterated truncated jet spaces:
	\begin{equation}
		\label{map:coplethysm-on-jet-spaces}
		\nu_{m,n}\:J^{m+n}(X) \to J^m(J^n(X))		
	\end{equation}
and hence in the limit
	$$
	\nu\:J(X) \to J(J(X))
	$$
This is the comultiplication map $J\to J \circ J$ for the comonad structure. The co-unit is given by
the projection $\jp\:J(X) \to J^0(X)=X$.
We also note that since $W$ is the left adjoint of $J$, it inherits a monad structure from
the comonad structure on $J$.

We define a $\dt$-structure on a sheaf $X$ to be a co-action of the comonad $J$ on $X$. This means a map $\alpha\:X\to J(X)$ such that the following diagrams commute:
\begin{equation}
\xymatrix{
X \ar^-\alpha[r]\ar_-\alpha[d] & J(X) \ar^-{J(\alpha)}[d] 
	& & X \ar^-\alpha[r]\ar_-{\mathrm{id}_X}[dr] & J(X) \ar^\jp[d]\\
J(X) \ar^-{\nu}[r] & J(J(X)) 
	& & & X
}\label{eqn:delta-structure}
\end{equation}
This definition could equivalently be given as an action of the monad $W$ on $X$ and we shall often use this definition instead.

If $X$ is affine, this definition of $\dt$-structure agrees with the one in section~\ref{sec:lambda-structure} above: there is a canonical bijection between the set of
$\dt$-structures on $\Spec(R)$ and the set of $\dt$-structures on $R$. The
category of $\dt$-sheaves and $\dt$-morphisms will be denoted by $\dtSh{}$.

Finally, if $X$ is a $\dt$-sheaf, writing $\varphi_X\:X\to X$ for the composition
		$$
		X \longlabelmap{\alpha} J(X) \longlabelmap{\varphi_{J(X)}} J(X) \longlabelmap{\jp} X.
		$$ 
we find that $\varphi_X\: X\to X$ is a lift of the Frobenius. (Note that when $X=J(Y)$, this map $\varphi_X$ 
agrees with the previously defined map $\varphi_{J(Y)}$.) Therefore,
every $\dt$-sheaf is equipped with a lift of the Frobenius and all $\dt$-morphisms are compatible with these
lifts.

\subsection{Relative $\dt$-structures}\label{sec:rel-lambda}
Let $S$ be a sheaf with a $\dt$-structure and denote by $\dtSh{S}$ the slice category of $\dtSh{}$ over $S$. 
Just as $\dtSh{}$ is comonadic over $\Sh$, the relative analogue 
$\dtSh{S}$ is comonadic over $\Sh_S$ (although no longer by definition). The comonad is
	$$
	X\mto J_S(X):= J(X) \times_{J(S)} S
	$$
where the map $S\to J(S)$ implicit in the product is the $\dt$-structure map $\alpha$ on $S$. Since both $\alpha$ and the map $J(X)\to J(S)$ are morphisms of $\dt$-sheaves, $J_S(X)$ inherits a canonical $\dt$-structure and is the fibre product in the category of $\dt$-sheaves. As in the absolute setting, we have truncations 
	$$
	J_S^n(X) = J^n(X)\times_{J^n(S)}S
	$$
and comultiplication maps
	\begin{equation}
		\label{map:relative-trunc-jet-copleth}
		\nu_{m,n} \: J_S^{m+n}(X) \longmap J_S^m J_S^n(X)
	\end{equation}
induced by the absolute comultiplication maps, which we abusively also denote by $\nu_{m,n}$.

It is also true that $\dtSh{S}$ is monadic over $\Sh$.
The monad is just $X\mto W(X)$ again, but here $W(X)$ is viewed as an object of $\dtSh{S}$ via the
composition 
$$
W(X)\longmap W(S) \longmap S
$$
via the adjunct $W(S)\to S$ of the $\dt$-structure map $\alpha$. So with some abuse of 
notation, we will still use $W$ to denote this monad on $\dtSh{S}$.

Given a $\dt_S$-sheaf $X$, the Frobenius lift $\varphi_X\: X\to X$ does not lie over the identity map on $S$
but over the Frobenius lift $\varphi_S\: S\to S$. Therefore, we obtain a relative Frobenius lift 
\begin{equation}
\label{map:rel-frob-lift}	
\varphi_{X/S}\:X\to \varphi_S^*(X)
\end{equation}
which is a morphism in $\dtSh{S}$.

The absolute coghost maps $\coghost_{n}\: J^n(S)\to S^{n+1}$ and $\coghost_{n}\: J^n(X)\to X^{n+1}$ 
induce relative coghost maps
	\[
	\coghost_{X/S, n}\: J_S^n(X)=J^n(X)\times_{J^n(S)}S\to X^{n+1}\times_{S^{n+1}}S=\prod_{i=0}^n \varphi_{S}^{\circ i *}(X),
	\]
and taking the limit of these maps, a relative coghost map in the infinite-length setting:
	\begin{equation}
		\label{map:rel-coghost-infinite-length}
	\coghost_{X/S}\: J_S(X)\to X^{\bN}\times_{S^{\bN}}S.	
	\end{equation}

Since the forgetful functor $\dtSh{S}\to \Sh_S$ is both monadic and comonadic, all limits and colimits
exist in $\dtSh{S}$, and the underlying $S$-sheaf of a limit or colimit can be computed in $\Sh_S$. In
particular, a group structure on an object $X\in\dtSh{S}$ is the same as a group structure on the
underlying $S$-sheaf such that the structure morphisms (multiplication, identity, and inverse) are
morphisms of $\dt_S$-sheaves.

\section{Further properties of $\dt$-structures and Witt vectors}

The purpose of this section is to establish some basic results about $\dt$-structures and Witt vectors.
They are only used in the proofs of the main theorems. None of them is needed to make sense of the theorems
themselves. So this section can safely be skipped and referred back to as needed.

\subsection{Relative $\dt$-structures are $\dt$-local with respect to the base}
Let $S$ be a $\dt$-sheaf and let $X$ be an $S$-sheaf. Then for each $\dt_{S}$-sheaf $S'$ we can consider the set \[\dt_{X/S}(S'):=\{\dt_{S'}\text{-structures on }X\times_S S'\}.\] If $S''\to S'$ is a morphism of $\dt_{S}$-sheaves and $X\times_S S'$ is equipped with a $\dt_{S'}$-structure then the fibre product $(X\times_S S')\times_{S'} S''=X\times_S S''$ is a $\dt_{S''}$-sheaf, or in other words, we are given a $\dt_{S''}$-structure on $X\times_S S''$. This gives a map \[\dt_{X/S}(S')\to \dt_{X/S}(S'')\] which makes the assignment $S'\mto \dt_{X/S}(S')$ a presheaf on $\dtSh{S}.$

\begin{lemm}\label{lemm:lambda-structures-etale} The functor $\dt_{X/S}$ is a sheaf (for the canonical topology on $\dtSh{S}$).
\end{lemm}
\begin{proof} Let $S''\to S'$ be an epimorphism of $\dt_S$-sheaves and write $S'''=S''\times_{S'} S''$.
	Then
	\[
	\xymatrix{X\times_{S}S''' \ar@<-.5ex>[r]\ar@<.5ex>[r] & X\times_S S'' \ar[r]&X\times_S S'}
	\]
is a coequaliser diagram in the category of $S$-sheaves.

If $X\times_S S'''$ is given a $\dt_{S'''}$-structure and $X\times_S S''$ is given a $\dt_{S''}$-structure such that the two maps $X\times_S S'''\rightrightarrows X\times_S S''$ are $\dt_{S''}$-maps then the $S'$-sheaf underlying the coequaliser of the two maps $X\times_S S'''\rightrightarrows X\times_S S''$ in the category of $\dt_{S'}$-sheaves is $X\times_S S'$, because $\dtSh{S'}$ is comonadic over $\Sh_{S'}$. Hence,  $X\times_S S'$ admits a unique $\dt_{S'}$-structure making the map $X\times_S S''\to X\times_S S'$ a $\dt_{S'}$-morphism. This is equivalent to the functor $\dt_{X/S}$ being a sheaf for the canonical topology on $\dtSh{S}$. 
\end{proof}

\subsection{$\dt$-structures on $p$-adic sheaves}
Let us say that a sheaf $S\in\Sh$ is \emph{$p$-adic} if the structure map $S\to \Spec(\Z)$ 
factors through \[\Spf(\Z_p)=\colim _n\Spec(\Z/p^n)\subset \Spec(\Z).\] For example, an affine scheme 
$\Spec(R)$ is $p$-adic if and only if $p$ is nilpotent in $R$.

\begin{prop}\label{prop:ghost-nil-immersion} If $S$ is a $p$-adic scheme, then so is $W_n(S)$ and the natural maps \[W_{n}(S)\to W_{n+m}(S)\] are nilpotent immersions.
\end{prop}
\begin{proof} 	
Using theorem~\ref{theo:witt-et-fibre},
we may pass to an affine cover of $S$ and therefore assume that $S=\Spec(R)$, where $p$ is
nilpotent on $R$. So we are reduced to showing
that the kernels of the truncation maps $W_{n+m}(R)\to W_n(R)$ are nilpotent.
It is further enough to take $m=1$. Then the kernel $VW_{n}(R)$ of $W_{n+1}(R)\to W_n(R)$ satisfies
$(VW_n(R))^2\subset pW_n(R)$, and so it is enough to show that $p$ is nilpotent in $W_n(R)$. However, as $R$ is
a $\Z/p^{i+1}\bZ$-algebra for some $i$, and $\Z/p^{i+1}\bZ=W_{i}(\F_p)$ it follows that $W_n(R)$ is a
$W_{n}(W_{i}(\F_p))$-algebra. Therefore it is enough to show that $p$ is nilpotent in $W_{n}(W_{i}(\F_p))$.
But the comonad comultiplication map \[W_{n+i}(\F_p)\to W_n(W_{i}(\F_p))\] makes $W_n(W_{i}(\F_p))$ an
algebra over the ring $W_{n+i}(\F_p)=\bZ/p^{n+i+1}\bZ$, in which $p$ is indeed nilpotent. 
\end{proof}

For any sheaf $S$ let $\Et_S$ denote the category of relatively representable \'etale algebraic spaces over $S$.
If $S$ also has a $\dt$-structure write $\Et^{\dt}_{S}$ for the category of relatively representable 
\'etale $\dt_S$-sheaves.

\begin{prop}\label{prop:et-witt-equivalence} If $S$ is a $p$-adic ind-scheme, the functor 
	\[  
	\Et^{\dt}_{{W(S)}} \longmap \Et_S, \quad Z\mapsto Z\times_{W(S)} S
	\] 
is an equivalence of categories with quasi-inverse $W$.
\end{prop}
\begin{proof}
First, observe that we may assume $S$ is a $p$-adic scheme because if $S=\colim_i S_i$, the categories in
question are the $2$-limits of the categories $\Et_{S_i}$ and $\Et^{\dt}_{{W(S_i)}}$ respectively.

Now the functor in question factors:
	$$
	\Et^{\dt}_{W(S)} \hookrightarrow \Et_{W(S)} \to \Et_S.
	$$
The second functor is an equivalence because of the equivalences \[\Et_{W_n(S)}\isomto \Et_{S} \] 
for all $n\geq 0$, which follow from proposition~\ref{prop:ghost-nil-immersion}.
Further $W\:\Et_S\to \Et_{W(S)}$
is a right quasi-inverse of the functor $\Et_{W(S)} \to \Et_S$,
by theorem~\ref{theo:witt-et-fibre}. But $W$ has essential image in the
subcategory $\Et^{\dt}_{W(S)}$. Therefore $\Et^{\dt}_{W(S)} \hookrightarrow \Et_S$
is an equivalence and $W$ is a two-sided quasi-inverse.
\end{proof}

\begin{prop}\label{prop:fin-loc-lambda} Let $S=\Spec(A)$, where $A$ is a ring in which $p$ is nilpotent.
Then the category of finite locally free $\dt_{W(S)}$-schemes is equivalent to the category of finite locally 
free $\dt_{W(A)}$-algebras.
\end{prop}
\begin{proof} 
As $W(A)=\lim_n W_n(A)$ and the transition maps are surjective with nilpotent kernels,
by proposition~\ref{prop:ghost-nil-immersion},
it follows that the categories of finite locally free $W(S)$-schemes and
finite locally free $W(A)$-schemes are equivalent, with the $W(A)$-scheme $\Spec(B)$ corresponding to the
$W(S)$-scheme $\colim_n \Spec(W_n(A)\otimes_{W(A)}B)$ and the $W(S)$-scheme $T$ corresponding to
the $W(A)$-scheme $\Spec(\lim_nB_n)$, where $\Spec(B_n)=T\times_{W(S)}W_n(S)$. (We leave the argument to the 
reader.)

Now, to give a $\dt_{W(A)}$-structure on $B$ is equivalent to giving a $W(A)$-morphism $B\to W(B)$ satisfying
certain properties. However, this is also equivalent to giving a family of morphisms 
	\[
	B_{m+n}\to W_n(B_m)
	\] 
for $n, m\geq 0$ which in turn is equivalent to giving morphisms 
	\[
	W_n(T_m)\to T_{n+m}
	\] 
for $n, m\geq 0$ whose colimit defines a $W(S)$-morphism 
	\[
	W(T)\to T.
	\] 
Tracing through these equivalences, we
see that to give a $\dt_S$-structure on $T$ is equivalent to giving a $\dt_{W(A)}$-structure on $B$.

\end{proof}

\begin{lemm}\label{lemm:lambda-hom-etale} 
	Let $S$ be a $p$-adic sheaf and let $\wt{G}$ and $\wt{H}$ be two 
	$\dt_{W(S)}$-groups over $W(S)$. If $\wt{H}$ is relatively representable and \'etale over $W(S)$, then the 
	natural map  \[\Hom_{\dt_{W(S)}}(\wt{H}, \wt{G})\to \Hom_S(H, G)\] from 
	$\delta_{W(S)}$-homomorphisms to $S$-homomorphisms is bijective.
\end{lemm}
\begin{proof}
	Any map $f\:H\to G=\wt{G}\times_{W(S)}S$ lifts by adjunction to a unique
$\dt_{W(S)}$-map $\wt{f}\:W(H)\to \wt{G}$. However, $\wt{H}=W(H)$ and the functor $W$ commutes with \'etale fibre
products (by proposition~\ref{prop:et-witt-equivalence}), from which it follows that $\wt{f}\: \wt{H}\to \wt{G}$ is a homomorphism
if and only if $f\: H\to G$ is a homomorphism. 
\end{proof}

\subsection{The sufficiency of $J^1$ for ind-affine sheaves}
Putting a $\dt$-ring structure on a ring is equivalent to putting an action of the comonad $W$ on it.
This is true nearly by definition, namely our definition of $W$. The purpose of this section is to prove the analogous result for ind-affine sheaves, where it appears to require proving something.

We will also work in the relative setting. So fix a $\dt$-sheaf $S$. Let $\dtind{S}$ denote the full
subcategory of $\dtSh{S}$ consisting of the objects which are ind-affine. Let $\dtoind$ denote the category of
pairs $(X,\beta)$, where $X$ is an ind-affine sheaf over $S$ and $\beta\:X\to J^1_S(X)$ is a section of the
projection $\jp\:J^1_S(X)\to X$ and where a morphism $(X,\beta)\to(X',\beta')$ is a morphism $f\:X\to X'$ of
$\Sh_S$ compatible with the sections in the evident sense: $\beta'\circ f = J^1_S(f)\circ \beta$.

Consider the functor 
\begin{equation}
\label{map:J1-sufficient}
	\dtind{S} \longmap \dtoind, \quad 	X \mapsto (X,\alpha_1)	
\end{equation}
where $\alpha_1$ is the composition
		$$
		\alpha_1\:X\longlabelmap{\alpha} J_S(X)\longlabelmap{\jp} J^1_S(X)
		$$
and where $\alpha$ is the structure map of the $\dt_S$-structure on $X$.
	
\begin{prop}\label{pro:W1-enough}
 The functor (\ref{map:J1-sufficient}) is an equivalence of categories. 
\end{prop}

One might say that $W$, viewed as a monad on ind-affine sheaves, is freely generated by the pointed functor
$W_1$. Before we prove it, we will need some preliminary results.

\begin{lemm}
	\label{lem-witt-induction}
	For any $n\geq 1$ and any ring $R$, the diagram
		$$
		\displaylabelfork{W_{n+1}(R)}{\Delta}{W_nW_1(R)}{\trunc}{\Delta\circ W_n(\trunc)}{W_{n-1}W_1(R)}
		$$
	is an equaliser diagram.
\end{lemm}
\begin{proof}
	This is immediate in Buium--Joyal components. Consider a Witt vector
		$$
		z=\big((x_{0},y_{0}),(x_1,y_1),\dots,(x_{n},y_{n})\big)\in W_nW_1(R).
		$$
	Then we have
		$$
		\trunc(z) = \big((x_{0},y_{0}),(x_{1},y_{1}),\dots,(x_{n-1},y_{n-1})\big)
		$$
	and
		$$
		\Delta\circ W_n(\trunc)(z) = \Delta(x_{0},\dots,x_{n}) 
			= \big((x_{0},x_{1}),(x_{1},x_{2}), \dots, (x_{n-1},x_{n})\big).
		$$
	These two Witt vectors are equal if and only if
		$$
		y_{0}=x_{1},\quad y_{1}=x_{2}, \quad \dots, \quad y_{n-1} = x_{n}.
		$$
	But by the explicit description of $\Delta\:W_{n+1}(R)\to W_nW_1(R)$ given in (\ref{map:coplethysm}),
	this is precisely the condition for $z$ to lie in the image of $\Delta$ and, second, $\Delta$ is 
	a monomorphism. 
\end{proof}

\begin{rema}
The map $\Delta\:W_{n+1}(R)\to W_nW_1(R)$ even has a functorial retraction. In ghost components, it
is given by 
		$$
		\big\langle \langle a_0,b_0\rangle,\dots,\langle a_n,b_n\rangle \big\rangle \mapsto
		\langle a_0,a_1,\dots,a_{n-1},a_n, b_n\rangle.
		$$	
\end{rema}

\begin{prop}\label{pro:jet-equalizer}
	If $X$ is ind-affine, the diagram 
		$$
		\displaylabelfork{J_S^{n+1}(X)}{\nu_{1,n}}{J_S^1 J_S^n(X)}{J_S^1(\jp)}{\nu_{1,n-1}\circ\jp}{J_S^1 J_S^{n-1}(X)}
		$$
	is an equaliser diagram.
\end{prop}
\begin{proof}
	For any affine $S$-scheme $\Spec(R)$, the diagram of $R$-points is
		$$
		\displaylabelfork{X(W_{n+1}(R))}{\Delta}{X(W_nW_1(R))}{\trunc}{\Delta\circ 
			W_n(\trunc)}{X(W_{n-1}W_1(R))}.
		$$
	It is enough to show that this is an equaliser diagram.
	If $X$ is affine, this is an immediate consequence of lemma~\ref{lem-witt-induction}.
	If $X$ is ind-affine, the result follows formally from the affine case and the fact
	that for filtered colimits we have
	$$(\colim_i X_i)(C)=\colim_i \big(X_i(C)\big)$$ 
	for any ring $C$ and the fact that finite limits are preserved by
	filtered colimits.
\end{proof}

\begin{proof}[Proof of proposition \ref{pro:W1-enough}]
First consider the category $\dtoind'$ consisting of pairs $(X,\beta)$, where $X\in\Sh_S$ and $\beta$ is any 
morphism $X\to J_S^1(X)$, and where a morphism $(X,\beta)\to (X',\beta')$ is any morphism $X\to X'$ compatible 
with $\beta$ and $\beta'$. (This is the category of so-called co-algebras for the functor $J_S^1$.)
Then $\dtoind$ is a full subcategory of $\dtoind'$, and
$\dtoind'$ is comonadic over $\Sh_S$ with comonad $X\mapsto \prod_n (J_S^1)^{\circ n}(X)$.
Therefore it is enough to show that for any object	
$(X,\beta)$ of $\dtoind$ and for any $n\geq 1$, the composition
	$$
	\beta_n\:X\longlabelmap{\beta} J_S^1(X) \longlabelmap{J_S^1(\beta)} J_S^1 J_S^1(X) \longlabelmap{}
	 \cdots\longlabelmap{} (J_S^1)^{\circ n}(X)
	$$
factors through the monomorphism
	$$
	\nu_n\: J_S^n(X) \longmap (J_S^1)^{\circ n}(X)
	$$
obtained by iterating the maps 
$\nu_{1,m-1}\:J_S^m(X)\to J_S^1 J_S^{m-1}(X)$ starting at $m=n$ and going down to $m=2$.
We do this by induction, the case $n=1$ being clear.
So assume $\beta_{n}$ factors through $\nu_{n}$, yielding a map $\alpha_{n}\:X\to J_S^n(X)$.

First, let $\beta'_{n+1}$ denote the composition
	$$
	X \longlabelmap{\beta} J_S^1(X) \longlabelmap{J_S^1(\alpha_n)} J_S^1 J_S^n(X),
	$$
and let
	$$
	\nu'_{n+1}\: J_S^{n+1}(X) \longmap J_S^1 J_S^n(X)
	$$
denote the comultiplication map $\nu_{1,n}$.
When composed with the monomorphism $J_S^1(\nu_n)\: J_S^1J_S^n(X)\to (J_S^1)^{\circ (n+1)}(X)$, the map $\beta'_{n+1}$ becomes $\beta_{n+1}$ and $\nu'_{n+1}$ becomes $\nu_{n+1}$.
Therefore we have $\nu'_{n}\circ\alpha_n = \beta'_{n}$, since $\nu_n\circ\alpha_n=\beta_n$.

Now, to prove that $\beta_{n+1}$ factors through $\nu_{n+1}$, 
it is certainly enough to prove that $\beta'_{n+1}$ factors through $\nu'_{n+1}$.
This will then follow from proposition~\ref{pro:jet-equalizer} by the universal property of equalisers, 
once we verify the equation
	$$
	{J_S^1(\jp)}\circ (J_S^1(\alpha_{n})\circ \beta) = {\nu'_n\circ\jp}\circ (J_S^1(\alpha_{n})\circ \beta )
	$$
But this holds because both sides equal $\beta'_{n}$:
	$$
	{J_S^1(\jp)}\circ J_S^1(\alpha_{n})\circ \beta = J_S^1(\alpha_{n-1}) \circ \beta = \beta'_{n}
	$$
and
	$$
	\nu'_n\circ\jp\circ J_S^1(\alpha_{n})\circ \beta = \nu'_n\circ \alpha_{n}\circ\jp\circ \beta
	 = \nu'_n\circ \alpha_{n} = \beta'_{n}.
	$$
\end{proof}

\subsection{$\dt$-structures on ind-affine sheaves} 
Let $S$ be a ind-affine sheaf. We say that $S$ is \emph{$p$-torsion free} if it is of the form $\colim_i
\Spec(R_i)$, where the pro-abelian group $(R_i[p])_{i\in I}$ is isomorphic to the zero pro-abelian group, where
$R_i[p]$ denotes the group of $p$-torsion elements in $R_i$. This is equivalent to saying that for any $i\in
I$, there exists a $j\geq i$ such that the map $R_j[p]\to R_i[p]$ is zero.

\begin{prop}\label{prop:p-torsion-free-affine-frobenius-lambda} 
The forgetful functor from the category of $p$-torsion free ind-affine $\dt$-sheaves to the category of
$p$-torsion free ind-affine sheaves equipped with a lift of the Frobenius is an equivalence.
\end{prop}
\begin{proof} 
Let $S=\colim_{i\in I} \Spec(R_i)$, where $(R_i)_{i\in I}$ is a $p$-torsion free pro-ring.
First we show that the functor is essentially surjective,
that any Frobenius lift on $S$ comes from a unique $\dt$-structure on $S$.

Let $\overline{W}_1(R_i)\subset R_i\times R_i$ denote the image of the ghost map $W_1(R_i)\to R_i\times R_i$. 
So we have $\overline{W}_1(R_i)=\{\langle x_0,x_1\rangle \in R_i^2 \mid x_1\equiv x_0^p\bmod pR_i\}$. Therefore
sections of the projection $\overline{W}_1(R_i)\to R_i$ are in natural bijection to Frobenius lifts on $R_i$.

We now show that this also holds for pro-rings. Suppose we have a Frobenius lift $\varphi$ on $(R_i)_{i}$.
So given any index $i\in I$, there is an index $j\geq i$ and a map $\varphi_{i,j}\:R_j\to R_i$ such that
for all $x\in R_j$, we have 
$\varphi_{i,j}(x)\equiv \bar{x}^p \bmod pR_i$, where $\bar{x}$ denotes the image of $x$ under the structure map
$R_j\to R_i$. Therefore the ring homomorphism $R_j\to R_i\times R_i$ defined by 
$x\mapsto \langle \bar{x},\varphi_{i,j}(x)\rangle$ has image in $\overline{W}_1(R_i)$, and hence on the pro-rings it defines a map $s\:(R_i)_i\ \to (\overline{W}_1(R_i))_i$ which is a section 
of the projection onto the left-hand component. Observe further that
its composition $(R_i)_i\to (\overline{W}_1(R_i))_i \to (R_i)_i$ with
the projection onto the right-hand component is indeed the given Frobenius lift $\varphi$ and that $s$ is the
only section with this property, since $\overline{W}_1(R_i)$ is a subset of $R_i\times R_i$. In this way,
sections of the projection $(W_1(R_i))_i\to (R_i)_i$ are in natural bijection to Frobenius lifts on $(R_i)_i$.

Since the analogous result for $\dt$-structures and $W_1$ is true by proposition~\ref{pro:W1-enough},
all that remains is to show 
that the map of pro-rings $(W_1(R_i))_i\to (\overline{W}_1(R_i))_i$ is a pro-isomorphism.
So for any $j\geq i$, consider the commutative diagram of exact sequences:
	$$
	\xymatrix{
	0 \ar[r] & R_j[p] \ar^V[r]\ar[d] & W_1(R_j) \ar^{\ghost_1}[r]\ar[d] 
		& \overline{W}_1(R_j) \ar[r]\ar[d]\ar@{-->}[dl] 
		& 0 \\
	0 \ar[r] & R_i[p] \ar^V[r] & W_1(R_i) \ar^{\ghost_1}[r] & \overline{W}_1(R_i) \ar[r] & 0,
	}
	$$
where $V(x)=(0,x)$.
As $(R_i)_i$ is $p$-torsion free, for any index $i$ 
there exists an index $j\geq i$ such that the map $R_j[p]\to R_i[p]$ is zero. It 
follows that for such $j$ there exists a unique map $\overline{W}_1(R_j)\to W_1(R_i)$ such that both the 
triangles in the diagram above commute. It therefore defines a morphism of pro-rings 
$(\overline{W}_1(R_i))_i\to (W_1(R_i))_i$ which is an inverse
of the map $(W_1(R_i))_i\to (\overline{W}_1(R_i))_i$ on both the left and the right.
\end{proof}

\begin{rema}\label{rema:p-tor-free-extra} If $S=\colim_{i}\Spec(R_i)$ is a $p$-torsion free ind-affine sheaf then: \begin{enumerate}[label=(\roman*)]
\item $W(S)$ is a $p$-torsion free ind-affine sheaf. Indeed, it is enough to show that
the pro-ring  $(W_n(R_i))_{i\in I}$ is $p$-torsion free for each $n$, but this is true because the kernel of the $n$-th ghost map \[\ghost_n\: W_n(R)\to \Pi_n(R)\] is the set of Witt vectors with Witt components $(x_0, x_1, \ldots, x_n)$ such that $p^i x_i=0$ for $0\leq i \leq n.$
(Note that it does not matter whether we use the Buium--Joyal components or the Witt components here. This is because the statement admits a coordinate-free formulation. It is equivalent to the element lying in the kernel of the map $W_n(R)\to W_j(R/(\text{$p^j$-torsion}))$ for all $j\leq n$. The proof, however, does seem
to be easier in the Witt coordinates.)
\item If $X\to S$ is relatively representable, flat and affine then $X$ is a $p$-torsion free ind-affine sheaf. Indeed, if $X=\colim \Spec(B_i)$, where $\Spec(B_i)=X\times_S \Spec(R_i)$, then each $B_i$ is flat over $R_i$ and so $B_i[p]=R_i[p]\otimes_{R_i}B_i$ and hence for each $i$ there is a $j\geq i$ such that $R_i[p]\otimes_{R_i}B_i \to R_i[p]\otimes_{R_i}B_i$ is zero.
\end{enumerate}
\end{rema}

\begin{prop}\label{prop:lambda-quotients} Let $S=\colim S_i$ be an ind-affine $\dt$-sheaf, let $X\to S$ be a relatively affine $\dt_S$-sheaf and let $E$ be a $\delta_S$-equivalence relation on $X$ in the category of \textup{fpqc}-sheaves.\footnote{We note that as $S$ and $W(S)$ are ind-affine sheaves it follows that they are sheaves on $\Aff$ for the fpqc topology and that all relatively affine sheaves (for the \'etale topology) over them are also sheaves on $\Aff$ the fpqc topology.} If the equivalence relation $E$ is one of the following:
\begin{enumerate}[label=\textup{(\alph*)}]
\item smooth,
\item finite locally free,
\end{enumerate} then the quotient $X/E$ in the category of \textup{fpqc}-sheaves over $S$ admits a unique $\delta_S$-structure compatible with the quotient map $X\to X/E$.
\end{prop}
\begin{proof} (a) If the equivalence relation is smooth then the quotient $X/E$ in the category of fpqc sheaves coincides with the quotient in the category of \'etale sheaves which then coincides with the quotient in the category of $\delta_S$-sheaves.

(b) Now assume that the equivalence relation is finite locally free. Write $E_i=E\times_S S_i$, $X_i=X\times_S S_i$, and $X/E$ and $X_i/E_i$
for the quotients in the category of fpqc sheaves. Then $X/E\times_S S_i=X_i/E_i$ is an affine scheme over $S_i$
(by usual descent, cf.\ Corollaire 7.6 Expos\'e VIII \cite{SGA1}) and hence $X/E$ is a relatively affine scheme over $S$. It remains 
to show that $X/E$ admits a unique $\dt_S$-structure for which the map $X\to
X/E$ is a $\dt_S$-map.

As the functor $W_n$ preserves equalisers of rings and $E \rightrightarrows X$ is a $\dt_S$-equivalence relation, 
for each $i,n$, there is a $j\geq i$ and a diagram of solid arrows
\[
\xymatrix{
W_{n}(E_i)\ar[d]\ar@<.5ex>[r]\ar@<-.5ex>[r]\ar[d] 
	& W_{n}(X_i)\ar[r]\ar[d] 
	& W_{n}(X_i/E_i)\ar@{.>}[d]\\
E_j\ar@<.5ex>[r]\ar@<-.5ex>[r] 
	& X_j\ar[r] 
	& X_j/E_j}
\] 
whose rows are coequalisers in the category of 
(absolutely) ind-affine sheaves over $S$ and which is commutative in the evident sense. Therefore, we obtain a unique morphism $W_{n}(X_i/E_i)\to X_j/E_j$, as shown, compatible with the rest of the diagram. Moreover, the diagrams above are compatible as $i$ and $j$ vary, so that we may take the colimit to obtain
a morphism $\rho_{n}\:W_n(X/E)\to X/E$. These morphisms are compatible as $n$ varies and hence
define a morphism $\rho\:W(X/E)\to X/E$.

To show that $\rho$ is a $\dt_S$-structure on $X/E$, it is enough to check that the $\rho_n$ satisfy a certain
associativity property (which will be recalled below). To do this, consider the diagram
\[
\xymatrix{W_n (W_m(E))
\ar@<.5ex>[r]\ar@<-.5ex>[r]\ar@<-.5ex>[d]\ar@<+.5ex>[d] & W_n(W_m(X))\ar[r]\ar@<-.5ex>[d]\ar@<+.5ex>[d] & W_n(W_m(X/E))\ar@<-.5ex>[d]\ar@<+.5ex>[d]
\\
W_{n+m}(E)
\ar[d]\ar@<.5ex>[r]\ar@<-.5ex>[r]\ar[d] & W_{n+m}(X)\ar[r]\ar[d] & W_{n+m}(X/E)\ar[d]\\
E\ar@<.5ex>[r]\ar@<-.5ex>[r] & X\ar[r] & X/E.}
\] 
defined as follows.
The horizontal morphisms are the obvious ones, the lower vertical maps 
$\rho_{m+n}\:W_{m+n}(\ast)\to\ast$ come from the $\dt_S$-structure or, in the case of $\ast=X/E$,
were constructed above.
The vertical parallel pairs
are given, on the left, by the coplethysm map (\ref{eq-truncated-comonad}), and on the right by
\[
W_{n}(W_m(\ast))\stackrel{\text{inclusion}}{\to} W_{n+m}(W_m(\ast))\stackrel{W_{n+m}(\rho_m)}\to
W_{n+m}(\ast).
\] 
The associativity property we need to show is that the two compositions down the right column agree
for every $m$ and $n$.

However, the two compositions down each of the other columns agree because $X$ and $E$ are $\delta_S$-sheaves. 
Additionally, the rows are still coequalisers in the category of absolutely ind-affine sheaves over $S$, as they 
are colimits of such.
Then because the diagram commutes (in the evident sense), it follows from a diagram chase that the map 
$\rho_{m+n}\:W_{n+m}(X/E)\to X/E$ is equalised by the pair of maps above it.
\end{proof}

\subsection{$\delta$-structures on $p$-torsion free schemes} 
The aim of this section is to prove that a Frobenius lift on a $p$-torsion free scheme
is equivalent to a
$\delta$-structure. Note that this does not automatically reduce to the affine case as the Frobenius lift in
general will not act as the identity on the underlying topological spaces. 

\begin{lemm} \label{lem:frob-lift-to-delta}
Let $X$ be a scheme (or algebraic space) equipped with a lift of the Frobenius $\varphi: X\to X$. Then the two 
compositions 
\[
\xymatrix{
X\times_{\Spec(\Z)} \Spec(\Z/p^{n+1}\Z)\ar[r]^-{i} 
& X \ar@<.5ex>[r]^-{w_{n+1}}\ar@<-.5ex>[r]_-{w_n\circ \varphi} 
& W_{n+1}(X)
}
\] 
agree.
\end{lemm}
\begin{proof}
Because the diagram is functorial in $X$,
we may assume that $X$ lies over $\Spec(\Z/p^{n+1}\Z)$, in which case we are reduced to showing that the maps 
\[
\xymatrix{
X \ar@<.5ex>[r]^-{w_{n+1}}\ar@<-.5ex>[r]_-{w_n\circ \varphi} 
& W_{n+1}(X)
}
\] 
agree. However, the Frobenius lift $\varphi: X\to X$ restricts to a unique Frobenius lift on
any affine \'etale neighbourhood $\Spec(A)\to X$. So we may reduce further to the case where $X=\Spec(A)$ is 
affine.

We now let $(a_0, \ldots, a_{n+1})\in W_{n+1}(A)$ in Witt coordinates and just compute:
\begin{eqnarray*}
\varphi(w_n(a_0, \ldots, a_{n+1})) &=& \varphi(a_0^{p^n}+\cdots + p^n a_n)\\
&=& \varphi(a_0)^{p^n}+\cdots + p^n \varphi(a_n)\\
&=& a_0^{p^{n+1}}+\cdots+p^n a_n^p\\
&=& w_{n+1}(a_0, \ldots, a_n), 
\end{eqnarray*}
where the last two equalities hold because $A$ is a $\Z/p^{n+1}\Z$-algebra.
\end{proof}

\begin{prop}\label{prop:delta-strucures-frobenius-lift} The forgetful functor from the category of $p$-torsion free algebraic spaces equipped with a $\delta$-structure to the category of $p$-torsion free algebraic spaces equipped with a lift of the Frobenius is an equivalence.
\end{prop}

\begin{proof} 
We will show essential surjectivity first. Let $X$ be a $p$-torsion free algebraic space and let $\varphi: X\to X$ be a lift of the Frobenius. We will show that there is a unique $\delta$-structure on $X$ inducing $\varphi.$ 
	
We will first prove that for each $n\geq 0$ there is a unique map 
	$
	\mu_n: W_n(X)\to X
	$ 
such that $\ghost_i\circ\mu_n=\varphi^{\circ i}$ for $0\leq i\leq n.$ 
For $n=0$, we take $\mu_0=\id_X$. Then by induction, we may construct $\mu_{n+1}$, assuming
$\mu_n$ has been constructed.
So it is enough to prove there exists a unique map 
$\mu_{n+1}: W_{n+1}(X)\to X$ such that $\mu_{n+1}|_{W_n(X)}=\mu_n$ and 
$\mu_{n+1}\circ \ghost_{n+1}=\varphi^{\circ (n+1)}$. 

Now consider the diagram
\[
\xymatrix{
X\times \Spec(\Z/p^{n+1}\Z)\ar@<-.5ex>[r]_-{i_2\circ i}\ar@<.5ex>[r]^-{i_1\circ\bar{\ghost}_{n+1}}
& W_n(X)\amalg X \ar[d]_{\mu_n\amalg \varphi^{\circ n+1}}\ar[r]
& W_{n+1}(X)\ar@{.>}[dl]^{\mu_{n+1}}\\
& X,
}
\]
where $i_1$ and $i_2$ denote the summand inclusions, $i$ is as in lemma~\ref{lem:frob-lift-to-delta},
and $\bar{\ghost}_{n+1}$ denotes
the reduced $(n+1)$-st ghost map of section 17.1 in~\cite{Borger:BGWV-II}. 
The two properties we require $\mu_{n+1}$ to satisfy are equivalent to the commutativity of the
right-hand triangle. Thus it is enough to prove the existence
and uniqueness of the map $\mu_{n+1}$ as shown making the triangle commute.

The top row is however a coequaliser in the category of algebraic spaces by theorem 17.3 
of~\cite{Borger:BGWV-II}.
So the existence and uniqueness of our map $\mu_{n+1}$ is equivalent to the vertical map 
$\mu_n\amalg \varphi^{\circ n+1}$ coequalising the two left horizontal arrows. In other words,
it is equivalent to the equality $\mu_n\circ\bar{\ghost}_{n+1}=\varphi^{\circ(n+1)}\circ i$.
But this follows immediately
from the equality $\mu_n\circ \ghost_n=\varphi^{\circ n}$, which holds by induction, and
the equality $\bar{\ghost}_{n+1}=\ghost_n\circ \varphi\circ i$, which is the statement of
lemma~\ref{lem:frob-lift-to-delta}.

In the limit, the maps $\mu_n\:W_n(X)\to X$ induce a map $\mu\:W(X)\to X$. To show that this is in fact an
action of the comonad $W$, we must show that the two induced maps $W(W(X))\rightrightarrows X$ agree. However we
know that their compositions with the ghost maps $\bN\times\bN\times X\to W(W(X))$ agree. This is just the
associativity property of the action of $\bN$ on $X$ induced by the endomorphism $\varphi$. Therefore it is
enough to observe that the ghost morphisms $w_n: \amalg_{n+1}X\to W_n(X)$ are epimorphisms in the category of
algebraic spaces. This holds because $\ghost_n$ is surjective and the induced map on structure sheaves
$\sO_{W_n(X)}\to \ghost_{n*}\sO_{\amalg_{n+1}X}$ is injective, which can be shown by applying the \'etale
localisation theorem~\ref{thm:w-etale} to reduce to $X$ affine and then invoking the fact that the ghost map
$W_n(R)\to R^{n+1}$ is injective and integral when $R$ is a $p$-torsion free ring. (The epimorphism property can
also be proved by the applying the coequaliser property in diagram above inductively.)

Therefore the functor in question essentially surjective. It is obviously faithful. It remains to show that it
is full. So let $X$ and $Y$ be two $p$-torsion free algebraic spaces equipped with lifts of the Frobenius, and
let $f: X\to Y$ be a morphism compatible with these lifts. Showing that $f$ is a $\delta$-morphism amounts to
showing that two maps $W(X)\rightrightarrows Y$ agree. However their compositions with the ghost map
$\ghost\:\bN\times X\to W(X)$ agree because $f$ is $\varphi$-equivariant. Now we can again invoke the fact that
the ghost maps $\ghost_n$ are epimorphisms when $X$ is a $p$-torsion free algebraic space 
to conclude that the original maps $W(X)\rightrightarrows Y$ do agree.
\end{proof}

\subsection{Affineness of the coghost map} We collect here a handful of affineness results concerning the coghost maps. The main application of these results is to prove certain uniqueness results for $\dt$-structures on abelian schemes.

\begin{lemm}\label{lemm:coghost-is-affine-finite-point-property} Let $X$ be a scheme with the property that every finite set of points of $X$ is contained in an open affine subscheme of $X$. Then for each $n\geq 0$ the coghost map \[\coghost_n\: J^n(X)\to X^{n+1}\] is affine. 
\end{lemm}
\begin{proof} The property satisfied by $X$ implies that there is an open affine cover 
$(X_i)_{i\in I}$ of $X$ such that $(X_i^{n+1})_{i\in I}$ is an open cover of $X^{n+1}$. However, the diagram \begin{equation}\label{dia:coghost-affine}
\begin{gathered}
	\xymatrix{
	J^n(X_i)\ar[r]^-{\coghost_{n}}\ar[d] & X^{n+1}_i\ar[d]\\
	J^n(X)\ar[r]^-{\coghost_n} & X^{n+1}
	}
\end{gathered}
\end{equation} is cartesian by proposition 12.2 of \cite{Borger:BGWV-II} (where it is also assumed that the open immersions $X_i\to X$ are closed but the proof that the diagram (\ref{dia:coghost-affine}) is cartesian does not use this assumption) and the top morphisms is affine. As $(X_i^{n+1})_{i\in I}$ is a cover of $X^{n+1}$ it follows that the bottom row of (\ref{dia:coghost-affine}) is affine.
\end{proof}

\begin{lemm}\label{lemm:coghost-affine-relative-coghost-affine} Let $S=\colim_{i\in I} S_i$ be an ind-affine sheaf equipped with a $\dt$-structure and let $f\: X\to S$ be sheaf over $S$. If, setting $X_i=X\times_S S_i$, the coghost maps \[\coghost_n\: J^n(X_i)\to X_i^{n+1}\] are affine for all $i$ then the relative coghost map \[\coghost_{X/S, n}\: J^n_S(X)\to X^{n+1}\times_{S^{n+1}}S\] is affine.
\end{lemm}
\begin{proof} 

The morphism $\coghost_{X/S, n}$ is affine if and only if the morphisms \[\coghost_{X/S, n}\times_S S_i\] are
affine for each $i\in I$. Fixing such an $i$, since $S_i$
is quasi-compact and quasi-separated there is some $j\in I$ such that the composition $S_i\to S\to J^n(S)$ 
factors through $J^n(S_j)\to J^n(S)$. Therefore, we have
	$$
	J^n_S(X)\times_S S_i = J^n(X)\times_{J^n(S)} S_i = J^n(X_j)\times_{J^n(S_j)} S_i.
	$$
It also follows that $S_i\to S \to S^{n+1}$ factors through $S_j\to S^{n+1}$ and hence we also have
	$$
	X^{n+1}\times_{S^{n+1}} S_i = X_j^{n+1}\times_{S_j^{n+1}} S_i.
	$$
We can now express $\coghost_{X/S, n}\times_S S_i$ as the composition along the top
row in the diagram

\begin{equation*}
\xymatrix{
J^n(X_j)\times_{J^n(S_j)} S_i\ar[d]\ar[r] &  J^n(X_j)\times_{S_j^n} S_i\ar[r]\ar[d] &  X_j^n\times_{S_j^n} S_i\\
J^n(S_j)\ar^-\Delta[r] & J^n(S_j)\times_{S_j^n} J^n(S_j) & }
\label{eqn:coghost-affine}
\end{equation*} 
The square in this diagram is cartesian, which shows that the top left arrow is affine. Since the top right arrow is affine by hypothesis, their composition $\coghost_{X/S, n}\times_S S_i$ is affine, and hence so is $\coghost_{X/S, n}.$
\end{proof}

We end this section with an application of the affineness results above which will have applications to the existence and uniqueness of $\dt$-structures on abelian schemes.

\begin{prop}\label{prop:anti-affine-lambda} Let $S$ be a $\dt$-sheaf and suppose we have the following:
\begin{enumerate}[label=\textup{(\alph*)}]
\item $X$ is an anti-affine\footnote{An $S$-sheaf $X$ is anti-affine if every morphism from $X$ to a relatively affine $S$-sheaf factors through the structure map $X\to S$.} $\dt_S$-sheaf equipped with a $\dt_S$-point $S\to X$,
\item $G$ is a $\dt_S$-group such that the relative coghost map 
$$\coghost_{G/S}\: J_S(G)\to G^{\bN}\times_{S^{\bN}}S$$
of (\ref{map:rel-coghost-infinite-length}) is affine, and
\item $f\: X\to G$ is an $S$-pointed morphism.
\end{enumerate}
Then $f$ is a $\dt_S$-morphism if (and only if) it is compatible with the relative Frobenius lifts on $X$ and $G$.
\end{prop}
\begin{proof} 
Let the $\dt_S$-structures on $X$ and $G$ be given by the $S$-morphisms 
	\[
	\alpha_{X/S}\: X\to J_{S}(X)\quad \text{ and } \quad \alpha_{G/S}\: G\to J_{S}(G).
	\] 
Since $f\: X\to G$ commutes with the relative Frobenius maps on $X$ and $G$, the difference 
	\[
	g:=\alpha_{G/S}\circ f-J_S(f)\circ \alpha_{X/S}\: X\to J_S(G)
	\]
factors through the kernel of the relative coghost homomorphism $\coghost_{G/S}$.

As $X$ and $G$ are equipped with $\dt_S$-points so are $J_S(X)$ and $J_S(G)$, and the morphisms composing $g$ are $S$-pointed, hence so is $g$. However, the kernel of $\coghost_{G/S}$ is $S$-affine by hypothesis and $X$ is $S$-anti-affine, so that $g$ factors through the structure morphism $X\to S$:
\[X\to S\to J_S(G).\] As the morphism $g$ is $S$-pointed, it follows that $g=0$.
So we have $\alpha_{G/S}\circ f=J_S(f)\circ \alpha_{X/S}$ and thus $f$ is a $\dt_S$-morphism.
\end{proof}

\section{Group schemes, torsors and extensions with $\dt$-structures} 

Let $S$ be an object of $\Sh$, assumed to be ind-affine and $p$-adic beginning in section~\ref{sec:delta-torsors}.
The sole purpose of this section is to prove corollary~\ref{cor:ext-equivalence}, which relates extensions of
\'etale group schemes over $S$ to extensions of their lifts to $W(S)$.

\subsection{Background on torsors and extensions}\label{subsec:torsor-constructions} 
Let $G$ and $H$ be a pair of relatively flat and affine
commutative group schemes over $S$. We will write $e_G\: S\to G$ and $\mu_G\: G\times_S G\to G$ for the identity
and multiplication maps (and similarly for all other groups). Then $S$, $G$ and $H$ are all sheaves for the fpqc
topology, and we denote by $\Tor_S(G)$ the fibred category of $G$-torsors over $\Aff_S$ for the fpqc topology.
It is a stack for the fpqc topology. We recall the following constructions for torsors (for those over $S$, with
the obvious extension to any base).

\begin{enumerate}[label=(\roman*)]
\item If $f\: G\to H$ is a homomorphism, we have a morphism of stacks 
	\[
	\Tor_S(G)\to \Tor_S(H)\: X\mto X\otimes_{G}H:=(X\times_S H)/G
	\] 
where $G$ acts on $X\times_S H$ via the restriction of the action of $G\times_S H$ along 
	\[
	G\to G\times_S H: g\mto (g, f(g^{-1})).
	\]
\item If $G$ and $H$ are a pair of sheaves of abelian groups over $S$ then the morphism 
	\[
	\Tor_S(G\times_S H)\to \Tor_S(G)\times_S \Tor_S(H): X\mto (X/H, X/G)
	\] 
is an equivalence of stacks.
\item If $X$ is a $G$-torsor then the dual $X^\vee$ of $X$ is the $G$-torsor whose underlying sheaf is $X$ with 
$G$ acting via the inverse of the given action. It is canonically isomorphic to the construction $X\mto 
X\otimes_{G}G$ in (i) for the homomorphism $G\to G: g\mto g^{-1}.$
\item The monoidal structure on $\Tor_S(G)$ is given by
	\begin{align*}
	\Tor_S(G)\times_S \Tor_S(G)&\to \Tor_S(G)\\
	(X_1, X_2)&\mto X_1\otimes_G X_2:= (X_1\times_S X_2)/G
	\end{align*}
where the quotient by $G$ is via the restriction of the action of $G\times_S G$ along $G\mto G\times_S G: g\mto
(g, g^{-1})$ and the $G$-action on the quotient is via either of the (equal) residual actions on $X_1$ and
$X_2$. This monoidal structure is symmetric and for all pairs of $G$-torsors $X_1, X_2$ we have canonical
isomorphisms 
	\[
	\underline{\Hom}_{S}^{G}(X_1, X_2)\isomto X_2\otimes_{G} X_1^\vee.
	\] 
\end{enumerate}

\subsubsection{}\label{subsubsec:kernel-morphisms-torsors} 
Let $f\: G\to H$ be a homomorphism with relatively representable affine and flat kernel $\ker(f)=K$. We
define $\Tor_S[f]$ to be the fibred category whose sections are given by the category of pairs $(X, \rho)$ where
$X$ is a $G$-torsor and $\rho\: H\to X\otimes_{G}H$ is a trivialisation, i.e., an isomorphism of
$H$-torsors. There is a natural functor
	\[
	\Tor_S(K)\to \Tor_S[f]: X\mto (X\otimes_{K}G, \rho)
	\] 
where $\rho$ is the trivialisation induced by
	\[
	X\otimes_{K}H = (X\times_S H)/K=(X/K)\times_S H=H.
	\] 
This functor is an equivalence of stacks with quasi-inverse given by sending $(X, \rho)$ to the equaliser of the
two morphisms $X\rightrightarrows X\otimes_{G}H$ given by the compositions 
	\[
	X\stackrel{X\times_S e_{H}}{\longrightarrow} X\times_S H\to X\otimes_{G}H \quad \text{ and } 
	\quad X\to S\stackrel{e_H}{\to} H\stackrel{\rho}{\to} X\otimes_{G}H.
	\]

\subsubsection{}\label{subsec:extensions-torsors} An extension of $H$ by $G$ is a short exact sequence of commutative group schemes over $S$ \[0\to G\to E\to H\to 0.\] A morphism of extensions $E\to E'$ is a homomorphism $E\to E'$ compatible with the identity maps on $G$ and $H$. This yields the moduli stack of extensions $\Ext_S(H, G)$ over $\Aff_S.$


The group structure on $H$ induces on $\Aff_H$ the structure of a symmetric monoidal stack over $\Aff_S$, as
follows. The product of two affine $S$-schemes $T_1$ and $T_2$ is the Cartesian product $T_1\times_S T_2$
viewed as an object of $\Aff_H$ via the composition $T_1\times_S T_2\to H\times_S H \to H$ with the product
map. (So if $T_1=T_2=S$, then this reduces to the usual group law on $H(S)$.) This defines a monoidal structure on the category of affine schemes over $H$, and since it is compatible with
base change in $S$, it induces a monoidal structure on the stack $\Aff_S$.

So we may also consider the Hom stack
	\[
	\sHom^{\otimes}_S(H, \Tor_S(G))
	\]
whose fibre over $S$ is the category of symmetric monoidal morphisms
	\[
	e\:\Aff_H\to \Tor_S(G)
	\]
over $\Aff_S$.

There is a natural equivalence of stacks (cf.\ proposition 1.4.23 of expos\'e XVIII in \cite{SGA4.3})
	\[
	\Ext_S(H, G)\to \sHom^{\otimes}_S(H, \Tor_S(G))
	\] 
sending an extension 
	\[
	0\to G\to E\to H\to 0
	\] 
to the symmetric monoidal functor $\eta\: \Aff_H\to \Tor_S(G)$ whose value on $T\to H$ is the $G$-torsor
$E\times_H T$ (the symmetric monoidal structure on this functor is induced by the commutative group structure on
$E$). A quasi-inverse is given by sending a symmetric monoidal functor $\eta\: \Aff_H\to \Tor_S(G)$ to the
$G$-torsor $E:=\eta(\id_H\: H\to H)$, which inherits the structure of an extension of $G$ by $H$ via the
symmetric monoidal structure on $\eta$. Indeed, the addition law on $E$ is obtained via the composition
	\[
	E\times_S E \isomto p_1^*(E)\times_{(H\times_S H)} p_2^*(E)
		\to p_1^*(E)\otimes_{G} p_2^*(E)\isomto \mu^*_H(E)\to E,
	\] 
where $p_i$ are the projections $H\times_S H\to H$ and $\mu_H$ is the multiplication.
The rest we leave to the reader.

\subsection{Torsors and extensions with $\dt$-structures}
\label{sec:delta-torsors} 
Let $\wt{G}/W(S)$ be a flat and affine group scheme over 
$W(S)$ such that either of the following two conditions hold:
\begin{enumerate}[label=(\alph*)]
\item $\wt{G}$ is smooth over $W(S)$,
\item $\wt{G}$ is finite locally free over $W(S)$.
\end{enumerate}
Write $G$ for the pull-back $\wt{G}\times_{W(S)}S$ along the first ghost component. 
(Later we will primarily consider groups $\wt{G}$ which are canonical lifts of a given group $G$, 
but for now $\wt{G}$
can be an arbitrary lift satisfying the properties above, although this is some abuse of notation).

Now assume that $\wt{G}$ has a $\dt_{W(S)}$-structure. A $\dt_{W(S)}$-$\wt{G}$-torsor is a $\wt{G}$-torsor over
$W(S)$ equipped with a $\dt_{W(S)}$-structure compatible with the $\wt{G}$-action. We denote by
$\Tor_{\dt_{W(S)}}(\wt{G})$ the fibred category over $\Aff_S$ whose fibre over an affine $S$-scheme $T$ is the
groupoid of $\dt_{W(T)}$-$\wt{G}_{W(T)}$-torsors. The fibred category $\Tor_{\dt_{W(S)}}(\wt{G})$ is a stack for the
\'etale topology and the constructions (i)--(iv) of section~\ref{subsec:torsor-constructions} for usual torsors
work just as well for the fibred categories of $\dt_{W(S)}$-torsors using proposition
\ref{prop:lambda-quotients} combined with the assumptions (a) and (b) above.

We denote by $g_{\wt{G}}$ the symmetric monoidal functor
\begin{align*}
	\Tor_{\dt_{W(S)}}(\wt{G})&\longlabelmap{g_{\wt{G}}} \Tor_S(G)\\
	X&\mto X\times_{W(S)}S.
\end{align*}

\begin{prop}\label{prop:smooth-equiv} 
	The functor $g_{\wt{G}}$ is fully faithful. If $G/S$ is smooth, then it is an equivalence.
\end{prop}
\begin{proof} 
If $\wt{X}_1$ and $\wt{X}_2$ are a pair of $\dt_{W(S)}$-$\wt{G}$-torsors, then we have
\begin{eqnarray*}
\Hom_{\dt_{W(S)}}^{\wt{G}}(\wt{X}_1, \wt{X}_2)
&=&\Hom_{\dt_{W(S)}}^{\wt{G}}(\wt{G}, \wt{X}_2\otimes_{\wt{G}}\wt{X}_1^{-1})\\
&=& \Hom_{\dt_{W(S)}}(W(S), \wt{X}_2\otimes_{\wt{G}}\wt{X}_1^{-1}))\\
&=&\Hom_S(S, X_2\otimes_{G}X_1^{-1})\\
&=&\Hom^G_S(G, X_2\otimes_{G}X_1^{-1})\\
&=&\Hom^G_S(X_1, X_2)
\end{eqnarray*} 
whence the full faithfulness.

If $G$ is smooth, then with respect to the \'etale topology, every $G$-torsor is locally trivial so that the
fully faithful functor $g_{\wt{G}}$ is locally essentially surjective. As both $\Tor_{\dt_{W(S)}}(\wt{G})$ and
$\Tor_S(G)$ are stacks for the \'etale topology it follows that $g_{\wt{G}}$ is an equivalence. 
\end{proof}

\subsubsection{} 
Consider an exact sequence of relatively flat and affine $\dt_{W(S)}$-group schemes 
	\[
	0\to \wt{K}\to \wt{G}\stackrel{\wt{f}}{\to} \wt{H},
	\] 
and denote its pull-back along the first ghost component by 
	\[
	0\to K\to G\stackrel{f}{\to} H.
	\]
We define $\Tor_{\dt_{W(S)}}[\wt{f}]$ to be the fibred category over $\Aff_S$ whose fibre over an affine $S$-scheme
$T$ is the groupoid of pairs $(\wt{X}, \wt{\rho})$ where $\wt{X}$ is a $\wt{G}$-torsor over $W(T)$ and
$\wt{\rho}\: \wt{H}\to \wt{X}\otimes_{\wt{G}}\wt{H}$ is an isomorphism of $\dt_{W(S)}$-$\wt{H}$-torsors together
with the obvious notion of morphism. We denote by $g_{\wt{f}}$ the functor 
\begin{align*}
\Tor_{\dt_{W(S)}}[\wt{f}]&\longlabelmap{g_{\wt{f}}} \Tor_S[f]\\
(\wt{X}, \wt{\rho})&\mto (X, \rho),
\end{align*}
where $\rho$ denotes the pull-back of $\wt{\rho}$ along the first ghost component map $S\to W(S)$.

\begin{prop}\label{prop:lambda-f-tor} 
	If $g_{\wt{G}}$ is an equivalence of categories, then so is $g_{\wt{f}}$.
\end{prop}
\begin{proof} 
Let $(\wt{X}_1, \wt{\rho}_1)$ and $(\wt{X}_2, \wt{\rho}_2)$ be a pair of objects of 
$\Tor_{\dt_{W(S)}}[\wt{f}]$ over some affine $S$-scheme.
Given a morphism $\theta\: (X_1, \rho_1)\to (X_2, \rho_2)$, the full faithfulness of $g_{\wt{G}}$ implies that
there exists a unique $\wt{\theta}\: \wt{X}_1\to \wt{X}_2$ lifting $\theta$. The faithfulness of
$g_{\wt{H}}$ shows that $\wt{\theta}$ is in fact a morphism 
$\wt{\theta}\: (\wt{X}_1, \wt{\rho}_1)\to (\wt{X}_2, \wt{\rho}_2)$, and so $g_{\wt{f}}$ is fully
faithful.

Now let $(X, \rho)$ be a section of $\Tor_S[f]$. Since $g_{\wt{G}}$ is an equivalence, there exists a unique
$\wt{X}$ lifting $X$. As $g_{\wt{H}}$ is fully faithful, it follows that there exists a unique 
$\wt{\rho}\: \wt{H}\to \wt{X}\otimes_{\wt{G}}\wt{H}$ lifting $\rho$. Therefore $g_{\wt{f}}$ is essentially
surjective and hence an equivalence. 
\end{proof}

\begin{prop}
\label{prop:G-equiv-then-K}
 If $g_{\wt{G}}$ is an equivalence of categories, then so is 
	 \[
	 g_{\wt{K}}\: \Tor_{\dt_{W(S)}}(\wt{K})\to \Tor_S(K).
	 \]
\end{prop}
\begin{proof} 
As in \ref{subsubsec:kernel-morphisms-torsors}, we can define a functor 
	\[
	\Tor_{\dt_{W(S)}}(\wt{K})\to\Tor_{\dt_{W(S)}}[\wt{f}]: \wt{X}\mto (\wt{X}\otimes_{\wt{K}}\wt{G}, \wt{\rho}),
	\]
and as before it is an equivalence. We then obtain a $2$-commutative diagram
\begin{equation*}
\begin{gathered}
	\xymatrix{
	\Tor_{\dt_{W(S)}}(\wt{K})\ar[d]_{g_{\widetilde{K}}} \ar[r]^-\sim
		& \Tor_{\dt_{W(S)}}[\wt{f}]\ar[d]^{g_{\widetilde{f}}}\\ 
	\Tor_S(K) \ar[r]^-\sim & \Tor_S[f].
	}
\end{gathered}
\end{equation*} 
The top and bottom arrows are equivalences and the right arrow is an equivalence by
proposition~\ref{prop:lambda-f-tor}. So the left arrow is also an equivalence. 
\end{proof}

\subsubsection{}\label{sec:lambda-extensions} 
Let $\wt{G}$ be a flat affine $\dt_{W(S)}$-group scheme over $W(S)$ satisfying either of the conditions (a) and
(b) of  section \ref{sec:delta-torsors}, and let $H$ be an affine \'etale group scheme over $S$. 
Then proposition~\ref{prop:et-witt-equivalence} implies $\wt{H}:=W(H)$ is the unique affine \'etale $\dt_{W(S)}$-group
scheme lifting $H$.

Denote by $\Ext_{\dt_{W(S)}}(\wt{H}, \wt{G})$ the fibred category over $\Aff_{S}$ whose fibre over an affine
$S$-scheme $T$ is the groupoid of $\dt_{W(T)}$-extensions of $\wt{G}$ by $\wt{H}$, by which we mean short exact sequences (in the sense of sheaves in the fpqc topology) of $\dt_{W(T)}$-group schemes 
	\[
	0\to \wt{G}_{W(T)}\to \wt{E}\to \wt{H}_{W(T)}\to 0,
	\] 
with morphisms $\dt_{W(T)}$-morphisms inducing the identity on $\wt{G}$ and $\wt{H}$. Denote by
	\[
	\sHom_{S}^\otimes(H, \Tor_{\dt_{W(S)}}(\wt{G}))
	\] 
the fibred category whose fibre over an affine $S$-scheme $T$ is the groupoid of symmetric monoidal functors
(over $\Aff_S$) 
	\[
	\eta\: \Aff_{H_T}\to \Tor_{\dt_{W(S)}}(\wt{G}).
	\] 
Such a functor $\eta$ is determined up to equivalence by its value at the universal point
	$$
	\wt{E}=\eta(H_T\to H_T),
	$$ 
which is a $\dt_{W(H_T)}$-$\wt{G}$-torsor $\wt{E}$ over $W(H_T)$, together with the 
$\dt_{W(H_T\times_T H_T)}$-$\wt{G}$-isomorphism 
	\begin{equation} 
		W(p_1)^*(\wt{E})\otimes_{\wt{G}} W(p_2)^*(\wt{E})\isomto
		W(\mu)^*(\wt{E}),
		\label{dia:lambda-isom-extension}
	\end{equation}
coming from the monoidal structure on $\eta$.

By theorem~\ref{theo:witt-et-fibre}, we have 
	\[
	W(H_T\times_T H_T) = W(H_T)\times_{W(T)}W(H_T) = \wt{H}_{W(T)}\times_{W(T)}\wt{H}_{W(T)}.
	\] 
Under this identification $W(p_1)$, $W(p_2)$ and $W(\mu)$ correspond to the usual projections and multiplication
map $\wt{H}_{W(T)}\times_{W(T)}\wt{H}_{W(T)}\to \wt{H}_{W(T)}$ respectively, and the isomorphism
(\ref{dia:lambda-isom-extension}) corresponds to a 
$\dt_{\wt{H}_{W(T)}\times_{W(T)} \wt{H}_{W(T)}}$-$\wt{G}$-isomorphism 
	\[
	p_1^*(\wt{E})\otimes_{\wt{G}}p_2^*(\wt{E}) \isomto \mu^*(\wt{E}).
	\] 
This induces the structure of a $\dt_{W(T)}$-extension of $\wt{H}_{W(T)}$ by $\wt{G}_{W(T)}$ on $\wt{E}$ and so
(as for usual torsors) we obtain an equivalence of stacks 
	\[
	\sHom_S^\otimes(H, \Tor_{\dt_{W(S)}}(\wt{G})) \to \Ext_{\dt_{W(S)}}(\wt{H}, \wt{G}).
	\]

\begin{theo}\label{theo:ext-equivalence} If $g_{\wt{G}}$ is an equivalence, then \[\Ext_{\dt_{W(S)}}(\wt{H}, \wt{G})\to \Ext_S(H, G)\] is also an equivalence. 
\end{theo}
\begin{proof} 
We have the $2$-commutative diagram 
\[
\xymatrix{
\sHom^{\otimes}_S(H, \Tor_{\dt_{W(S)}}(\wt{G}))\ar[r]^-\sim\ar[d]_\wr 
	& \Ext_{\dt_{W(S)}}(\wt{H}, \wt{G}) \ar[d]\\ 
\sHom^{\otimes}_S(H, \Tor_S(G))\ar[r]^-\sim 
	& \Ext_S(H, G)
}
\] 
in which the left arrow (composition with the equivalence $g_{\wt{G}}$) is an equivalence. 
The bottom arrow is an equivalence, as explained in~\ref{subsec:extensions-torsors}. Similarly,
the top arrow is an equivalence, as just explained in~\ref{sec:lambda-extensions}. 
\end{proof}

\begin{rema}
	This theorem can be viewed as a derived analogue of lemma~\ref{lemm:lambda-hom-etale}.
\end{rema}

\begin{coro}\label{cor:ext-equivalence}
	Let $H$ be an affine \'etale group scheme over $S$, and let $\wt{G}$ be a flat and affine $\dt_{W(S)}$-group
	scheme satisfying \textup{(a)} or \textup{(b)} of section~\ref{sec:delta-torsors}.
	If $\wt{G}$ is isomorphic to the kernel of a $\dt_{W(S)}$-homomorphism of flat and affine $\dt_{W(S)}$-group 
	schemes	$\wt{G_1}\to\wt{G_2}$ with $G_1$ smooth over $S$, then the functor
		\[
		\Ext_{\dt_{W(S)}}(\wt{H}, \wt{G})\to \Ext_S(H, G)
		\]
	is an equivalence.
\end{coro}
\begin{proof}
	This follows immediately from proposition~\ref{prop:G-equiv-then-K}, proposition~\ref{prop:smooth-equiv}, 
	and theorem~\ref{theo:ext-equivalence}.
\end{proof}

\section{Canonical lifts of ordinary $p$-groups}

\subsection{$\dt$-structures on \'etale and multiplicative $p$-groups}
\label{sec:ordinary-groups} By a $p$-group over a $p$-adic sheaf $S$ we mean any finite locally free group scheme $G$ locally of $p$-power degree. We say that $G$ is \'etale if $G\to S$ is \'etale, and we say that $G$ is multiplicative if $G$ is the Cartier dual of an \'etale $p$-group.

By an ordinary $p$-group over $S$ we mean a $p$-group $E$ for which there exists a short exact sequence of 
sheaves of groups 
	\[
	0\to G\to E\to H\to 0
	\]
where $G$ is a multiplicative $p$-group scheme and $H$ an \'etale $p$-group scheme. (See 
remark~\ref{rema:ordinary} for an alternative characterisation.)

If $E$ is an ordinary $p$-group over $S$, then $G$ and $H$ are uniquely determined (as sub and quotient groups of $E$) and are compatible with base change. Indeed, the inclusion $G\to E$ is both an open and closed immersion (being the pull back of the open and closed immersion $S\to H$ along $E\to H$) and is therefore smallest open and closed sub-sheaf containing the identity section of $E$. We will write $G=E_\mul$ and $H=E_\et$ and note that the assignments $E\mto E_\mul$ and $E\mto E_\et$ are functorial in ordinary $p$-groups $E$ to multiplicative and \'etale $p$-groups respectively.

It follows from the above that if $E$ is a $p$-group over $S$ which is fpqc locally an ordinary $p$-group then it is itself an ordinary $p$-group and the fibred category $\Ord$ over $\Aff_{\Spf(\Z_p)}$, whose fibre over each $p$-adic affine scheme $S$ is the groupoid of ordinary $p$-groups, is a stack for the fpqc topology.

We note that every homomorphism $G\to E$, from a multiplicative $p$-group to an ordinary $p$-group, factors
through the inclusion $E_\mul\to E.$ In particular, there are no non-trivial homomorphisms from a multiplicative
$p$-group to an \'etale $p$-group.

\begin{rema}\label{rema:ordinary} 
A $p$-group $G$ over $S$ is ordinary in the sense above if and only if:
	\begin{enumerate}[label=(\roman*)]
	\item for all $s\in S$ the $p$-group $G_s$ is ordinary, and
	\item the degree of $(G_s)_{\mathrm{mult}}$ is a locally constant function on $S.$
	\end{enumerate}
As follows from lemma 4.8 of \cite{Messing72}. In particular, if $A/S$ is an abelian scheme such that $A_s$ is 
ordinary for all $s\in S$ then $A[p^n]$ is ordinary for all $n\geq 0.$
\end{rema}

\begin{prop}\label{prop:lambda-etale-mult-groups} 
Let $S$ be $p$-adic sheaf.
	\begin{enumerate}[label=\textup{(\roman*)}]
	\item Every \'etale (resp.\ multiplicative) $p$-group over $W(S)$ admits a unique $\dt_{W(S)}$-structure 
		compatible with its group structure.
	\item Every homomorphism of such groups is a $\dt_{W(S)}$-homomorphism.
	\item The category of \'etale (resp.\ multiplicative) $\dt_{W(S)}$-$p$-groups is equivalent, via base change 
		along the first ghost component, to the category of \'etale (resp.\ multiplicative) $p$-groups over $S$.
	\end{enumerate}
\end{prop}
\begin{proof} 
(i): For \'etale groups this follows from proposition \ref{prop:et-witt-equivalence}.

For multiplicative groups, we first note that there is one and only one way to lift a multiplicative group $G/S$
to a multiplicative group $\wt{G}/W(S)$ (forgetting about $\dt_{W(S)}$-structures). Indeed, the only possible
option is to take the Cartier dual of the unique lift of the Cartier dual of $G$ (which is an \'etale group). In
particular, it follows that every multiplicative $p$-group $\wt{G}/W(S)$ is, after base change along a morphism
$W(S')\to W(S)$ with $S'\to S$ \'etale and surjective, of the form 
\begin{equation} 
	\prod_{i=1}^r \mu_{p^{n_i}}.
	\label{eqn:group-form}
\end{equation} 
Using the \'etale local nature of $\dt_{W(S)}$-structures on sheaves over $W(S)$, as given by lemma
\ref{lemm:lambda-structures-etale}, we may assume that $\wt{G}$ is of the form (\ref{eqn:group-form}) and show
that it admits a unique $\dt_{W(S)}$-structure compatible with its group structure. It is enough to show
existence and uniqueness locally in the \'etale topology so that we may also assume that $S=\Spec(A)$ is affine.
By proposition \ref{prop:fin-loc-lambda} we can instead show that the $W(A)$-Hopf algebra 
\[
R:=\bigotimes_{i=1}^r W(A)[T]/(T^{p^{n_i}}-1)=W(A)[T_1, \ldots, T_r]/(T_1^{p^{n_1}}-1, \ldots, T_1^{p^{n_r}}-1)
\] 
admits a unique $\dt_{W(A)}$-structure compatible with its $W(A)$-Hopf algebra structure. It is clear that $R$
admits at least one such $\dt_{W(A)}$-structure given by $\delta_R(T_i)=0$ for $i=1, \ldots, r$. So we will show
that this is the only one.

Consider a $\dt_{W(A)}$-structure $\dt_R$ on $R$. Set 
$P=\Pi_{i=1}^r \{0,\ldots, p^{n_i}-1\}$ and for $I=(a_1, \ldots, a_r)\in P$ write $T^I=T_1^{a_1}\cdots T_r^{a_r}$
so that $\{T^I: I\in P\}$ forms a basis for $R$ as $W(A)$-module.
The operator $\delta_R$ is completely determined by its values 
	\[
	\delta_R(T_i)=\sum_{I\in P}a_{I, (i)} T^I
	\] 
on the elements $T_i$ for $1\leq i\leq r$. The compatibility of $\dt_R$
with the Hopf algebra structure on $R$ is by definition equivalent to the equality of set maps
\[
\Delta_R\circ\delta_R=\delta_{R\otimes_{W(A)} R}\circ \Delta_R,
\] 
which is in turn equivalent to the equalities for $1\leq i\leq r$: 
\[
\Delta_R(\delta_R(T_i)) = \delta_{R\otimes_{W(A)} R}(\Delta_R(T_i)).
\]
These equalities expand to
\[
\Delta_R(\sum_{I\in P}a_{I, (i)} T^I) = \delta_{R\otimes_{W(A)} R}(T_i\otimes T_i)
\]
and then, applying the product rule to $\delta_{R\otimes_{W(A)} R}((T_i\otimes 1)(1\otimes T_i))$, to
\[
\sum_{I\in P}a_{I, (i)} T^I\otimes T^I
=\sum_{I\in P} a_{I, (i)}\left(T^p_i\otimes T^I+T^I\otimes T^p_i\right)
+\sum_{I, J\in P} p a_{I,(i)} a_{J, (i)} T^I\otimes T^J.
\] 
Equating the coefficients of $T^I\otimes T^I$ we find that $a_{I,(i)}=p a_{I,(i)}^2$ unless 
$I = (0,\ldots, p, \ldots, 0)$ with $p$ in the $i$th position, where we find that 
$a_{I, (i)}=2a_{I,(i)}+p a_{I, (i)}^2.$ Regardless, we find that $a_{I, (i)}(1\pm pa_{I, (i)})=0$ which, as $p$
is topologically nilpotent, in $W(R)$ shows that $a_{I, (i)}=0.$ Hence $\delta_R(T_i)=0$ for $1\leq i\leq r$ and
the only $\dt_{W(A)}$-structure on $R$ compatible with its Hopf algebra structure is the one with $\delta(T_i)=0$.

(ii): By additivity we are reduced to showing that all $W(S)$-group scheme homomorphisms \[\mu_{p^m}\to \mu_{p^n}\] are $\dt_{W(S)}$-morphisms or, again using proposition~\ref{prop:fin-loc-lambda}, that every Hopf algebra homomorphism \[\theta\: W(A)[T]/(T^{p^n}-1)\to W(A)[T]/(T^{p^m}-1)\] is a $\dt_{W(A)}$-homomorphism. As $\delta(T)=0$ (on both sides) it is enough to show that $\delta(\theta(T))=0$.

As $\theta$ is a Hopf algebra homomorphism, we must have \[\theta(T)=\sum_{i=1}^{p^n}a_iT^i\] where the 
$a_i\in W(A)$ are orthogonal idempotents (i.e. $a_i a_j = \delta_{ij} a_i$ and $a_i^2=a_i$). As
$W(A)[T]/(T^{p^n}-1)$ is $p$-adically complete and the $a_i$ are orthogonal and idempotent, it follows from
lemma \ref{lemm:orth-idem-p-adic} below that 
\[
\delta(\theta(T))=\delta\left(\sum_{i=1}^{p^n}a_i T^i\right)
=\sum_{i=1}^{p^n}\delta(a_i T^i)=\sum_{i=1}^{p^n}a_i\delta(T^i)=0
\]
(as $\delta(T)=0$, it follows immediately that $\delta(T^i)=0$). Therefore, $\theta$ is a
$\dt_{W(A)}$-homomorphism. 

(iii): Since the functor is essentially surjective, by (i), it is enough to prove it is full and faithful. The
category \'etale $p$-groups over $W(S)$ is equivalent, by proposition \ref{prop:et-witt-equivalence}, to that
over $S$. The same is true for multiplicative $p$-groups by Cartier duality. Therefore for any \'etale (resp.\
multiplicative) $p$-groups $G$ and $H$ over $W(S)$, every homomorphism $G\times_{W(S)}S\to H\times_{W(S)}S$ lifts
to a unique homomorphism $G\to H$. By (ii), it is necessarily a $\dt_{W(S)}$-morphism. 
\end{proof}

\begin{lemm}\label{lemm:orth-idem-p-adic} Let $R$ be a $p$-adically complete $\delta$-ring.
\begin{enumerate}[label=\textup{(\roman*)}]
\item If $r\in R$ is idempotent, then $\delta(r)=0$ and $\delta_R(rr')=r\delta(r')$ for all $r'\in R$.
\item If $r_1, r_2\in R$ satisfy $r_1r_2=0$, then $\delta(r_1+r_2)=\delta(r_1)+\delta(r_2)$.
\end{enumerate}
\end{lemm}
\begin{proof} (i) As $\delta(rr')=r^p \delta(r')+r'^p\delta(r)+p\delta(r)\delta(r')$ and $r^p=r$ it is enough to show that $\delta(r)=0$. For this, note that $\delta(r)=\delta(r^2)=2r^p\delta(r)+p\delta(r)^2$ and hence $\delta(r)(2r-1+p\delta(r))=0$. But $2r-1$ is a unit as $(2r-1)^2=1$ and hence so is $2r-1+p\delta(r)$ as $R$ is $p$-adically complete. Therefore $\delta(r)=0.$

(ii) This follows immediately from the addition law for $\delta$.
\end{proof}

\subsubsection{} 
Given (for now) a multiplicative or \'etale $p$-group $G$ over $S$,
we refer to the unique $\dt_{W(S)}$-$p$-group over $W(S)$ lifting $G$ 
as its canonical lift. We will typically denote it $\wt{G}$.
This does conflict with our earlier convention of allowing $\wt{G}$ to denote any lift of $G$ to $W(S)$,
but we believe this will not cause any confusion.

\subsection{$\dt$-structures on ordinary $p$-groups}
\label{sec:ordinary-lambda-groups} 
By an ordinary $\dt_{W(S)}$-$p$-group we mean an ordinary $p$-group $\wt{E}$ over $W(S)$ equipped with a
$\dt_{W(S)}$-structure compatible with its group structure. In this case, the multiplicative and \'etale group
schemes $\wt{G}:=\wt{E}_\mul$ and $\wt{H}:=\wt{E}_\et$ are both uniquely $\dt_{W(S)}$-groups, by proposition \ref{prop:lambda-etale-mult-groups}. Further, the morphisms in the short exact sequence 
	\[
	0\to \wt{G}\to \wt{E}\to\wt{H}\to 0
	\] 
are $\dt_{W(S)}$-morphisms. Indeed, $\wt{G}\to \wt{E}$ is an open immersion (of $p$-adic sheaves),
and $\wt{E}\to \wt{H}$, being the
quotient of $\wt{E}$ by the finite locally free $\dt_{W(S)}$-group $\wt{G}$, is a
$\dt_{W(S)}$-morphism, by proposition~\ref{prop:lambda-quotients}. We denote by $\Ord_{\dt}$ the fibred category
over $\Aff_{\Spf(\Z_p)}$ whose fibre over a $p$-adic affine scheme $S$ is the groupoid of ordinary
$\dt_{W(S)}$-$p$-group schemes. The \'etale-local nature of $\dt_{W(S)}$-structures shows that this is a stack
for the \'etale topology.

\begin{lemm}
	\label{lemm:easy-case} 
Let $G$ and $H$ be a multiplicative and an \'etale $p$-group over $\Spf(\Z_p)$ respectively. Then base change 
along the first ghost component induces an equivalence:
\begin{equation}
\label{map:ext-equiv}	
\Ext_{\dt_{W(\Spf(\Z_p))}}(\wt{H}, \wt{G})\to \Ext_{\Spf(\Z_p)}(H, G).
\end{equation}
\end{lemm}
\begin{proof} 
Since the source and target of the functor are stacks,
it is enough to show equivalence locally on the base $\Spf(\bZ_p)$. Therefore we may assume $G=\prod_{i=1}^n \mu_{p^{r_i}}$ and hence
	\[
	\wt{G}=\prod_{i=1}^n \mu_{p^{r_i}/W(\Spf(\Z_p))}.
	\] 
Then $\wt{G}$ is the kernel of a homomorphism between relatively affine and smooth group schemes 
	\[
	0\to \wt{G}\to \mathbf{G}_{\mathrm{m}/W(\Spf(\Z_p))}^n\stackrel{(p^{r_1}, \ldots, p^{r_n})}{\longrightarrow} 
		\mathbf{G}_{\mathrm{m}/W(\Spf(\Z_p))}^n.
	\] 
Further if $\mathbf{G}_{\mathrm{m}/W(\Spf(\Z_p))}$ is given its usual $\dt_{W(S)}$-structure, these
maps are $\dt_{W(S)}$-equivariant, by proposition~\ref{prop:lambda-etale-mult-groups}.
Corollary~\ref{cor:ext-equivalence} then implies the functor (\ref{map:ext-equiv}) is an equivalence.
\end{proof}

\begin{prop}\label{prop:canonical-lifts-ordinary-groups} 
The functor 
	\[
	\Ord_{\dt}\to \Ord: \wt{E}/W(S)\mto E/S
	\] 
induced by pull-back along the first ghost component is an equivalence of groupoids.
\end{prop}
\begin{proof} 
First, as $G$ and $H$ vary over all multiplicative and \'etale $p$-groups over $\Spf(\Z_p)$ the maps 
	\[
	\big(\Ext_{\Spf(\Z_p)}(H, G)\to \Ord\big)_{(G, H)}
	\]
form a cover of $\Ord.$ Indeed, an ordinary $p$-group over a $p$-adic affine scheme $S$ is locally an extension
of a split \'etale group $H$ by a split multiplicative group $G$, and all such groups defined over $\Spf(\bZ_p)$.
The claim will then follow if we can show that for each pair $(G, H)$, the map $\mc{E}\to \Ext_{\Spf(\Z)}(H, G)$ 
is an equivalence where $\mc{E}$ is the fibre product 
	\[
	\xymatrix{
	\mc{E}\ar[r]\ar[d] & \Ext_{\Spf(\Z_p)}(H, G)\ar[d] \\
	\Ord_{\dt}\ar[r] & \Ord.
	}
	\] 

The sections of $\mc{E}$ over a $p$-adic affine scheme $S$ are triples 
	\[
	(\wt{E}, E', \rho\: \wt{E}\times_{W(S)}S\isomto E')
	\] 
where $\wt{E}$ is an ordinary $\dt_{W(S)}$-$p$-group and $E'/S$ is an extension of $G_S$ by $H_S$.
To show that \[\mc{E} \to \Ext_{\Spf(\Z_p)}(H, G)\] is an equivalence we will instead show that the natural map 
	\[
	\Ext_{\dt_{W(\Spf(\Z_p))}}(\wt{H}, \wt{G})\to \mc{E}: 
	\wt{E}\mto (\wt{E}, \wt{E}\times_{W(S)}S, \id_{\wt{E}\times_{W(S)}S})
	\] 
is an equivalence, from which the claim follows, as the composition 
	\[
	\Ext_{\dt_{W(\Spf(\Z_p))}}(\wt{H}, \wt{G})\to \mc{E} \to \Ext_{\Spf(\Z_p)}(H, G)
	\] 
is isomorphic to the equivalence (\ref{map:ext-equiv}).

If $(\wt{E}, E', \rho\: \wt{E}\times_{W(S)}S\isomto E')$ is a section of $\mc{E}$ over $S$, then we claim that
$\wt{E}$ admits a unique structure of an extension of $\wt{G}_{W(S)}$ by $\wt{H}_{W(S)}$ lifting the
corresponding structure on $\wt{E}\times_{W(S)}S$ induced by $\rho$. But this is obvious, as the isomorphisms
	\[
	\wt{E}_{\et}\times_{W(S)}S\isomto E'_\et = H_S \quad \text{ and } \quad \wt{E}_{\mul}\times_{W(S)}S\isomto
	E'_{\mul}=G_S
	\] 
induced by $\rho$ lift uniquely to isomorphisms 
	\[
	\wt{E}_{\et}\isomto \wt{H}_{W(S)} \quad \text{ and }\quad \wt{E}_{\mul}\isomto \wt{G}_{W(S)}
	\] 
by proposition \ref{prop:lambda-etale-mult-groups},
which gives $\wt{E}$ the structure of a $\dt_{W(S)}$-extension of
$\wt{H}_{W(S)}$ by $\wt{G}_{W(S)}$ and so the map in question is essentially surjective.

Moreover, if \[(\wt{f}, f')\: (\wt{E}_1, E'_1, \rho_1)\to (\wt{E}_2, E'_2, \rho_2)\] is any isomorphism of
sections of $\mc{E}$ over $S$ then it is easy to check that the induced map of $\dt_{W(S)}$-$p$-groups
	\[
	\wt{f}\: \wt{E}_1\to \wt{E}_2
	\]
is a morphism of the corresponding $\dt_{W(S)}$-extensions of $\wt{H}_{W(S)}$
by $\wt{G}_{W(S)}$ and so the functor in question is also full and therefore an equivalence.
\end{proof}

\begin{theo}\label{theo:ordinary-lambda} 
For each $p$-adic affine scheme $S$, the \textup{category} of ordinary $\dt_{W(S)}$-$p$-groups is equivalent, 
via base change along the first ghost component, to the category of ordinary $p$-groups over $S.$
\end{theo}
\begin{proof} 
By proposition~\ref{prop:canonical-lifts-ordinary-groups}, this functor induces an equivalence of the groupoids.
In particular, it is essentially surjective, and so we need only show that it is fully faithful.

First, the categories under consideration are exact, with admissible epimorphisms those which are epimorphisms
of fpqc sheaves, and the functor, pull-back along the first ghost component, is exact. Hence, fixing a pair
$\wt{E}$ and $\wt{E}'$ of ordinary $\dt_{W(S)}$-$p$-groups we obtain a morphism of exact sequences where the
rightmost column denotes the groups of Yoneda extensions:
\[
\xymatrix@C=10pt{
\Hom_{\dt_{W(S)}}(\widetilde{E}'_\et, \widetilde{E})\ar@{>->}[r]\ar[d] 
	& \Hom_{\dt_{W(S)}}(\widetilde{E}', \widetilde{E})\ar[r]\ar[d] 
	& \Hom_{\dt_{W(S)}}(\widetilde{E}'_\mul, \widetilde{E})\ar[d] \ar[r] 
	& \mathrm{Ext}_{\dt_{W(S)}}(\widetilde{E}'_\et, \widetilde{E})\ar[d]\\
\Hom_S(E'_\et, E)\ar@{>->}[r] 
	& \Hom_S(E', E)\ar[r] 
	& \Hom_S(E_\mul', E)\ar[r] 
	& \mathrm{Ext}_{S}(E_\et', E).
}
\]
By proposition \ref{lemm:lambda-hom-etale}, the first vertical arrow is bijective; by proposition
\ref{prop:lambda-etale-mult-groups}, the third arrow is bijective (noting that any $\dt_{W(S)}$-homomorphism
$\wt{E}'_\mul\to \wt{E}$ factors uniquely as a $\dt_{W(S)}$-homomorphism $\wt{E}'_\mul\to \wt{E}_\mul$, and any
homomorphism $E'_\mul\to E$ factors uniquely as a homomorphism $E_\mul\to E_\mul$); and by
corollary~\ref{cor:ext-equivalence}, the fourth vertical arrow is bijective. Hence by the Five Lemma, the second
vertical arrow is bijective and we are done. 
\end{proof}

\begin{rema}
	\label{rmk:ind-groups}
It follows from this theorem that the corresponding categories of ind-objects are equivalent---the 
the category of ordinary ind-$p$-groups is equivalent to that of ordinary ind-$\dt_{W(S)}$-$p$-groups. In 
particular, the category of
ordinary $\dt_{W(S)}$-$p$-divisible groups is equivalent to the category of ordinary $p$-divisible groups over
$S$. 

Therefore we have a canonical-lift functor for ordinary $p$-divisible groups over any base $S$.
This will allow us to define canonical lifts of ordinary abelian schemes using the theorem of Serre--Tate. 
\end{rema}

\section{Canonical lifts of ordinary abelian schemes}
We refer the reader to chapter 1 of Faltings--Chai~\cite{FaltingsChai} for a concise review of the 
theory of abelian schemes in the language we will use. We emphasise that by Raynaud's theorem 
(\cite{FaltingsChai}, p.\ 5), an abelian-scheme structure is \'etale-local on the base. So we may 
safely speak of abelian schemes over any sheaf.

\subsection{$\dt$-structures on abelian schemes} 
Before we consider canonical lifts of ordinary abelian schemes we will need several lemmas on $\dt$-structures
and abelian schemes over ind-schemes $S$ (which need not be $p$-adic).

\begin{lemm}\label{lemm:hom-ab-scheme-unram} Let $S$ be an ind-scheme, let $f\: A\to B$ be a homomorphism of abelian schemes over $S$ and let $S_0\to S$ be a nilpotent immersion. If $f\times_S S_0=0$ then $f=0$. In other words, $\underline{\Hom}_S(A, B)$ is formally unramified.
\end{lemm}
\begin{proof} 
We may assume that $S$ is an affine scheme. If $f\times_S S_0=0$ then it follows
that the homomorphism $f\: A\to B$ factors through the formal completion of the identity section of $B$: 
	$$
	A\to \widehat{B}\to B.
	$$
However, $\widehat{B}$ is the colimit of the infinitesimal neighbourhoods of the zero
section in $B$, each of which is affine over $S$, so that as $A$ is quasi-compact over $S$ the homomorphism
$A\to B$ factors through one of these infinitesimal neighbourhoods. But these neighbourhoods are all $S$-affine,
and as $A$ is an $S$-abelian scheme, it is anti-affine.
Therefore this a map factors further through the structure map $A\to S$. In other words, it is the zero
homomorphism.
\end{proof}

\begin{prop}\label{lemm:hom-abelian-lambda} 
Let $S$ be an ind-affine $\dt$-sheaf, and let $A$ and $B$ be a pair of abelian schemes over $S$ with
$\dt_S$-structures (not necessarily compatible with the group structures). Let $f\: A\to B$ be an
$S$-homomorphism. If $f$ commutes with the relative Frobenius lifts on $A$ and $B$, then it is a
$\dt_S$-morphism.
\end{prop}
\begin{proof} 
Write $S=\colim_{i\in I} S_i$ as a filtered colimit of affine schemes. By theorem 1.9 of chapter I of
\cite{FaltingsChai}, for each $i\in I$ the $S_i$-scheme $B\times_S S_i$ satisfies the hypotheses of
lemma~\ref{lemm:coghost-is-affine-finite-point-property} so that we may apply
lemma~\ref{lemm:coghost-affine-relative-coghost-affine} and deduce that the relative coghost homomorphism
	\[
	\coghost_{B/S, n}\: J_{S}^n(B)\to B^{n+1}\times_{S^{n+1}}S
	\]
is affine. Taking the limit over $n$, we see for general reasons that
	\[
	\coghost_{B/S}\: J_{S}(B)\to B^\bN\times_{S^\bN}S
	\]
is also affine.

At the same time, $A$ is a relative abelian scheme over $S$ and hence is $S$-anti-affine.
The claim then follows from proposition~\ref{prop:anti-affine-lambda}.
\end{proof}

\begin{coro}\label{coro:lambda-structures-on-abelian-varieties-are-determined-by-psi-structures} 
Let $S$ be an ind-affine $\dt$-sheaf, and let $A/S$ be a relative abelian scheme over 
$S$. Then the following hold:
\begin{enumerate}[label=\textup{(\roman*)}]
\item Two $\dt_S$-structures on $A$ are equal if and only if their relative Frobenius lifts $\varphi_{A/S}$
(defined in~(\ref{map:rel-frob-lift})) agree.
\item A $\dt_S$-structure on $A$ is compatible with the group law if and only if the relative Frobenius lift
$\varphi_{A/S}$ is a group homomorphism.
\end{enumerate}
\end{coro}
\begin{proof}
(i) Apply proposition \ref{lemm:hom-abelian-lambda} to the morphism $\id\: A\to A$, where the source has one
$\dt_S$-structure and the target has the other.

(ii) Apply proposition \ref{lemm:hom-abelian-lambda} to the group law $A\times_{S} A\to A$.
\end{proof}

\subsection{Canonical lifts} 
We now fix an ind-affine $p$-adic sheaf $S$. As pointed out in remark~\ref{rmk:ind-groups}, we
have a theory of canonical lifts for ordinary $p$-divisible groups over
$S$. This allows us to define canonical lifts of ordinary abelian schemes using the theorem of Serre--Tate,
which we recall below.

Let $S_0\to S$ be a nilpotent thickening of $p$-adic sheaves and denote by $D_{S_0/S}$ the category whose
objects are triples $(G, A_0, h)$ where $G/S$ is a $p$-divisible group, $A_0/S_0$ is an abelian variety and 
$h\: G\times_S S_0\isomto A_0[p^\infty]$ is an isomorphism. A morphism $(G,A_0,h)\to(G',A'_0,h')$ is a
homomorphism $G\to G'$ and a homomorphism $A_0\to A'_0$ compatible with $h$ and $h'$ in the evident sense.

\begin{theo}[Serre--Tate]\label{serre-tate} 
	The functor from the category of abelian schemes over $S$ to $D_{S_0/S}$ given by 
	\[
	A/S\mto (A[p^\infty], A\times_S S_0, \mathrm{id})
	\]
is an equivalence of categories.
\end{theo}
\begin{proof}
	When $S$ is affine, see 1.2.1 of \cite{Katz} or the appendix of~\cite{Drinfeld:symmetric-domains}
	or the original reference, theorem (2.3) on p.\ 166 of~\cite{Messing72}.
	The general case then follows for formal reasons.
\end{proof}

Now, if $A/S$ is an ordinary abelian scheme there exists a unique ordinary abelian scheme $\wt{A}/W(S)$ with the property that $\wt{A}[p^\infty]/W(S)$ is the unique ordinary $\dt_{W(S)}$-$p$-divisible group lifting $A[p^\infty]/S.$ We call $\wt{A}/W(S)$ the canonical lift of $A/S$.

It also follows from theorem \ref{serre-tate} that the abelian scheme $\wt{A}/W(S)$ admits a unique lift of the Frobenius \[\varphi_{\wt{A}/W(S)}\: \wt{A}\to \varphi^*_{W(S)}(\wt{A})\] which is a homomorphism.

\begin{theo} 
	\label{thm:can-lift}
	There is a unique $\dt_{W(S)}$-structure on $\wt{A}/W(S)$ compatible with its group structure. Moreover, $\wt{A}/W(S)$ is the unique deformation of $A/S$ admitting a $\dt_{W(S)}$-structure compatible with its group structure.
\end{theo}
\begin{proof} 
	As $\wt{A}/W(S)$ admits a lift of the relative Frobenius which is a group homomorphism, uniqueness of a
	$\dt_{W(S)}$-structure on $\wt{A}$ inducing this relative Frobenius lift follows from
	corollary \ref{coro:lambda-structures-on-abelian-varieties-are-determined-by-psi-structures},
	which also shows that such a $\dt_{W(S)}$-structure must be compatible with the group law.
	
	Moreover, any deformation of $A/S$ along $S\to W(S)$ which admits a $\dt_{W(S)}$-structure compatible with the group law induces a $\dt_{W(S)}$-structure on the corresponding $p$-divisible group, which must therefore be isomorphic to the $p$-divisible group of $\wt{A}/W(S)$ in a unique way, and so such a deformation must also be isomorphic to $\wt{A}$ in a unique way.
	
	It remains to show that there exists a $\dt_{W(S)}$-structure on $\wt{A}/W(S)$ inducing the given relative
	Frobenius lift. For this we may first assume that $S=\Spec(R)$ is affine. We can then write $R=\colim_i R_i$ as a filtered colimit of finitely presented $\Z_p$-algebras. By passage to the limit, any abelian scheme $A/\Spec(R)$ comes from an abelian scheme $A_i/\Spec(R_i)$ so that we may assume that $R$ is a finitely presented $\Z_p$-algebra. Then $R$ is of the form $\Z_p[T_1, \ldots, T_r]/I$ where $I$ is a finitely generated ideal containing some power of $p$. Since the deformation theory abelian schemes is unobstructed (see 8.5.24.(a) of \cite{Illusie}), we can find a compatible family of abelian schemes $A_n/\Spec(\Z_p[T_1, \ldots, T_r]/I^n)$ for $n\geq 1$ with $A_1=A$, or in other words, an abelian scheme $A'$ over the $p$-adic formal scheme $\colim_n \Spec(\Z_p[T_1, \ldots, T_r]/I^n)$. Note that $A'$ is ordinary as $A$ is and that $\colim_n \Spec(\Z_p[T_1, \ldots, T_r]/I^n)$ is a $p$-torsion free
	ind-affine sheaf. Therefore, in order to show that the relative Frobenius lift on the canonical lift of an ordinary abelian scheme comes from a $\delta_{W(S)}$-structure, it is enough to do so for ordinary abelian schemes defined over $p$-torsion free ind-affine schemes $S$ of the form $\colim_n \Spec(\Z_p[T_1, \ldots, T_r]/I^n)$ where $I$ contains some power of $p.$
	
	As the transition maps in the system $S=\colim_n \Spec(\Z_p[T_1, \ldots, T_r]/I^n)$ are nilpotent immersions, it
	follows that we can find a representable open cover $(\wt{A}_i\to \wt{A})_i$ where each $\wt{A}_i$ is a
	relatively affine and smooth $W(S)$-scheme. In particular they are $p$-torsion free ind-affine sheaves, as
	explained in remark~\ref{rema:p-tor-free-extra}. The Frobenius lift $\varphi_{\wt{A}/S}$ restricts to each
	$\wt{A}_i$ (because $\wt{A}$ is $p$-adic and the Frobenius map is topologically the identity) and equips each
	$\wt{A}_i$ with a unique $\dt_{W(S)}$-structure by
	proposition~\ref{prop:p-torsion-free-affine-frobenius-lambda}. Because of this uniqueness, the
	$\dt_{W(S)}$-structures glue to give a $\dt_{W(A)}$-structure on $\wt{A}$ itself which by construction induces
	the given Frobenius lift. 
\end{proof}

\begin{theo}\label{theo:can-lift-p-adic} Let $S$ be a $p$-adic scheme. Then the category of ordinary $\dt_{W(S)}$-abelian schemes is equivalent, via base change along the first ghost component, to the category of ordinary abelian schemes over $S$.
\end{theo}
\begin{proof} By theorem \ref{thm:can-lift} every ordinary abelian scheme over $S$ can be lifted to an ordinary abelian scheme over $W(S)$ equipped with a compatible $\delta_{W(S)}$-structure. Therefore, the functor in question is essentially surjective. Moreover, it is faithful by lemma \ref{lemm:hom-ab-scheme-unram}, as $S\to W(S)$ is a nilpotent immersion.

For fullness, let $f\: A\to B$ be a homomorphism of ordinary abelian schemes over $S$. The restriction of $f$
to the associated $p$-divisible groups lifts to a unique $\dt_{W(S)}$-morphism between the associated canonical
lifts (of the $p$-divisible groups). By the Serre--Tate theorem (\ref{serre-tate}), it follows that $f$ itself
lifts to a unique morphism \[\wt{f}\: \wt{A}\to \wt{B}.\] As the restriction of $\wt{f}$ to the $p$-divisible
groups is a $\dt_{W(S)}$-homomorphism it follows that \[\wt{f}\circ \varphi_{\wt{A}}-\varphi_{\wt{A}'}\circ
\varphi^*_{W(S)}(\wt{f})\: \wt{A}\to \varphi^*_{W(S)}(\wt{B})\] is trivial on the $p$-divisible groups and is
therefore trivial itself. Hence $\wt{f}$ commutes with the Frobenius lifts on $\wt{A}$ and $\wt{B}$ is a
$\dt_{W(S)}$-homomorphism, by lemma \ref{lemm:hom-abelian-lambda}.
\end{proof}

\subsection{Iterated canonical lifts}
\label{subsec:iterated}
Since $W(S)$ is also $p$-adic, we may iterate the canonical lift construction.
However it follows from theorem~\ref{thm:can-lift} that there is a unique isomorphism
	\[
	\wt{\wt{A}} \isomto W(W(S))\times_{W(S)} \wt{A}
	\]
of $\dt_{W(W(S))}$-abelian schemes lifting the identity map on $A$, where the implicit morphism $W(W(S))\to W(S)$ in the product is the colimit of the maps $W_n(W_m(S))\to W_{m+n}(S)$ induced by (\ref{eq-truncated-comonad}).

One should view this as expressing the fact that the canonical lift functors equip moduli stack of ordinary abelian schemes over $p$-adic with a $\delta$-structure, in the following sense.

If $\mc{M}$ denotes the moduli stack in question, then we may formally take the jet space $J(\mc{M}):=\mc{M}\circ W$: an $S$-point of $J(X)$ is an ordinary abelian scheme over $W(S)$. The canonical lift functor $A/S\mto \wt{A}/W(S)$ is then a morphism \[\alpha: \mc{M}\to J(\mc{M})\] such that the composition with the projection $u: J(\mc{M})\to \mc{M}$ is canonically isomorphic to the identity of $\mc{M}$. The compatibility of iterated canonical lifts described above, then says that the left square below also commutes up to canonical isomorphism: \[
\xymatrix{
	\mc{M} \ar^-\alpha[r]\ar_-\alpha[d] & J(\mc{M}) \ar^-{J(\alpha)}[d] 
	& & X \ar^-\alpha[r]\ar_-{\mathrm{id}_\mc{M}}[dr] & J(\mc{M}) \ar^\jp[d]\\
	J(\mc{M}) \ar^-{\nu}[r] & J(J(\mc{M})) 
	& & & \mc{M}
}\]

The reader may now compare with diagram (\ref{eqn:delta-structure}) of \ref{subsec:delta-sheaves} defining $\delta$-structures on sheaves. Thus, the structure on $\mc{M}$ above is clearly a stack theoretic incarnation of that. However, the authors warn that $\delta$-structures on stacks certainly require more than the simple commutativity (up to canonical isomorphism) of the diagrams above, and exactly which approach is best in this case will be left to others to consider.

Thus, with the provisos mentioned above, we may view the canonical lift functor $A\mapsto \wt{A}$ as
a $\dt$-structure on the $\Spf(\bZ_p)$-stack of ordinary abelian varieties, even though
it is not algebraic when $g>1$.

\section{Duality and algebraisation}
\label{sec:polar}

\subsection{Duality and polarisations} 
Let $S$ be a $p$-adic sheaf and let $A/S$ be an ordinary abelian scheme. It is easily checked that kernel of the
relative Frobenius 
\[
\varphi_{\wt{A}/W(S)}: \wt{A}\to \varphi_{W(S)}^*(\wt{A})
\]
is $\wt{A}[p]_{\mathrm{mult}}\subset \wt{A}[p]$ so that we can factor the multiplication-by-$p$ map as
\[
\wt{A}\stackrel{\varphi_{\wt{A}/W(S)}}{\longrightarrow}
\varphi_{W(A)}^*(\wt{A})\stackrel{v_{\wt{A}/W(S)}}{\longrightarrow} \wt{A}.
\]
It is immediate from the definition that the homomorphism 
\[
v_{\wt{A}/W(S)}: \varphi^*_{W(S)}(\wt{A})\to \wt{A}
\] 
lifts the relative Verschiebung homomorphism. The dual of the relative Verschiebung is the relative
Frobenius, and hence the dual of $v_{\wt{A}/W(S)}$ 
\[
v_{\wt{A}/W(S)}^\vee: \wt{A}^\vee\to \varphi_{W(S)}^*(\wt{A}^\vee)
\]
defines a lift of the the relative Frobenius on the dual abelian scheme $\wt{A}^\vee$.

\begin{prop} The canonical lift functor $A\mto \wt{A}$ is compatible with duality.
That is, there is a unique isomorphism $\wt{A}^\vee \isomto \wt{A^\vee}$
lifting the identity map on $A^\vee$.
\end{prop}
\begin{proof} 
By the uniqueness of theorem \ref{thm:can-lift}, it is enough to show that for an abelian scheme $A/S$, the dual
$\wt{A}^\vee$ of the canonical lift admits a $\delta_{W(S)}$-structure compatible with its group structure. As
the canonical lift and the dual are compatible with base change we may, as in the proof of theorem
\ref{thm:can-lift}, assume that $S$ is a $p$-torsion free ind-affine sheaf. We now need only show that
$\wt{A}^\vee$ admits a lift of the relative Frobenius which is compatible with its group structure, but by the
remarks above we may take $\varphi_{\wt{A}^\vee/W(S)}:=v_{\wt{A}/W(S)}^\vee.$ 
\end{proof}

\begin{coro}
\label{cor:polaristion}
Let $\lambda\:A\to A^{\vee}$ be a polarisation. Then its canonical lift 
	\[
	\tilde{\lambda}\:\wt{A}\longmap \wt{A^\vee}=(\wt{A})^{\vee}
	\]
is a polarisation, locally of the same degree.
\end{coro}
\begin{proof}
This holds because being a polarisation is a property of the fibres $\lambda_s$ over geometric points $s$
of $S$, and $W(S)$ is a ind-nilpotent thickening of $S$. 
\end{proof}

\begin{rema}
As in section~\ref{subsec:iterated}, we may therefore view the canonical lift functor $A\mapsto \wt{A}$ as
a $\dt$-structure on the algebraic $\Spf(\bZ_p)$-stack of ordinary abelian varieties of dimension $g$ and with 
a polarisation of degree $d^2$. Because the stack is algebraic, this could no doubt be constructed more 
directly, as we did in~\cite{Borger-Gurney:Nagoya}.
\end{rema}

\subsection{Algebraisation}
Often the canonical lift of an abelian variety over a perfect field $k$ is understood to be an abelian
scheme over the affine scheme $\Spec(W(k))$ instead of the formal scheme $\Spf(W(k))$.
The following algebraisation theorem extends this to
canonical lifts of families parametrised by a $p$-adic formal affine schemes $S=\colim_m \Spec(R_m)$ whose 
transition maps are nilpotent immersions. Let us fix this notation for the remainder of this section,
along with $R=\lim_m R_m.$

\begin{theo}\label{theo:algebraisation}
Let $A$ be a polarised ordinary abelian scheme over $S$. Then 
the canonical lift $\wt{A}$, which is a polarised abelian scheme over the ind-scheme $W(S)$, is 
algebraisable to a polarised abelian scheme $\wt{A}^{\mathrm{alg}}$ over the affine scheme $\Spec(W(R))$.
It is unique up to unique isomorphism.
Moreover, $\wt{A}^{\mathrm{alg}}$ admits a unique lift of the Frobenius extending that on $\wt{A}$.
\end{theo}

If we regard the canonical lift $\tilde{A}$ as a map from $W(S)$ to the moduli stack of polarized abelian
varieties, the theorem says that this map prolongs to $\Spec(W(R))$, uniquely up to unique isomorphism. We will
prove it by invoking algebraisation results available in the literature, namely those in Bhatt's paper
\cite{Bhatt:algebraization}.

\begin{proof} 
Let $d^2$ denote the degree of the given polarisation of $A$. For any $N\geq 1$, let $\sA_N$ denote the
moduli stack of abelian varieties of dimension $\dim_S(A)$ with full level-$N$ structure and a polarisation of
degree $d^2$. It follows from geometric invariant theory (theorem 7.9 of~\cite{Mumford:GIT3}, p.\ 139) that
$\sA_N$ is quasi-projective for sufficiently large $N$. Now fix such an $N$ which is further not a multiple of
$p$. Then $\sA_N$ is not only quasi-projective but also finite \'etale and $G$-Galois over the stack
$\sA_1=\sA/G$, where $G$ denotes the group controlling change of level structure. The existence of the
algebraisation $\wt{A}^{\mathrm{alg}}$ now follows from lemma~\ref{lem:alg-point-stack} below.

We now show that the Frobenius lift extends to $\wt{A}^{\mathrm{alg}}$. First, if $f: A\to B$ is a finite locally free homomorphism of ordinary abelian schemes over $S$ and $A^{\alg}$ and $B^\alg$ denote their algebraisations then there is a unique $f^{\alg}: A^\alg\to B^\alg$ whose pull back to $S$ is $f$.
Indeed, we may assume that the degree of $f$ is constant and equal to $N$, as the degree of is locally constant of $S$ (i.e.\ for each $N\geq 1$, the sub-sheaf $S_N\subset S$ where $\deg(f)=N$ is both open and closed). Therefore, if $A[f]\subset A$ denotes the kernel of $f$ then $A[f]\subset A[N^2]$ and so $A[f]$ algebraises uniquely to some $A[f]^{\alg}\subset A^{\alg}[N]$ by the equivalence of categories of finite locally free schemes over $S$ and $\Spec(W(R))$. It follows that $A^{\alg}/A^{\alg}[f]$ is an algebraisation of $B=A/A[f]$ and hence there is a unique isomorphism $A^{\alg}/A^{\alg}[f]\isomto B^{\alg}$ and $f^{\alg}$ is the composition: \[A^{\alg}\to A^{\alg}/A^{\alg}[f]\isomto B^{\alg}.\]

Therefore, returning to the case at hand, we find a unique map
\[\varphi_{\wt{A}^\mathrm{alg}}: \wt{A}^{\mathrm{alg}}\to
\varphi^*(\wt{A}^{\mathrm{alg}}),\]
whose pull-back to $S$ is the Frobenius lift $\varphi_{\wt{A}}$ of $\wt{A}$. We now show that $\varphi_{\wt{A}^\mathrm{alg}}$ is in fact a Frobenius lift.

Let us write $W(S)^\alg=\Spec(W(R))$, $W(S)^\alg_p=W(S)^\alg\times_{\Spec(\Z_p)}\Spec(\F_p)$, $\wt{A}^\alg_{p}=\wt{A}^\alg\times_{W(S)^\alg}W(S)^\alg_p$ and $\varphi_{\wt{A}^\alg_p}=\varphi_{\wt{A}^\alg}\times_{W(S)^\alg}W(S)^\alg_p$.

 The closed immersion $W(S)_p^\alg\to W(S)^\alg$ induces a bijection on closed points, and each closed point of $W(S)^\alg$ factors through $W(S)\to W(S)^\alg$. It follows that $\varphi_{\wt{A}^\alg_p}$ agrees with $\Fr_{\wt{A}^\alg_p}$ at each closed point. Therefore, $\ker(\varphi_{\wt{A}^\alg_p}-\Fr_{\wt{A}^\alg_p})\to \wt{A}^\alg_p$ is a nilpotent immersion. As $\wt{A}^\alg_{p}\to \Spec(W(R)/p)$ is smooth, it follows that $h=\varphi_{\wt{A}^\alg_p}-\Fr_{\wt{A}^\alg_p}$ is equal to the zero morphism after base change along the reduction $(W(S)^\alg_p)_{\mathrm{red}}\to W(S)^\alg_p$. Finally, as $h$ is finitely presented it follows that there is a nilpotent closed immersion $Z\to W(S)^\alg_p$ such that $h\times_{W(S)^\alg_p} Z=0$ and it now follows by rigidity that $h=\varphi_{\overline{A}}-\Fr_{\overline{A}}=0$ and therefore $\varphi_{\wt{A}^\alg}$ is a lift of the Frobenius.
\end{proof}

\begin{lemm}\label{lem:alg-point-stack} Let $\Gamma\rightrightarrows X$ be a finite \'etale groupoid in schemes with $X$ quasi-compact and quasi-separated over $\Spec(\bZ_p)$. If $Y$ denotes the quotient stack $X/\Gamma$, then the canonical functor 
		\begin{equation*}
		Y(W(R)) \longisomap \lim_{m,n} Y(W_n(R_m))	
		\end{equation*}
		is an equivalence.
\end{lemm}
\begin{proof} For any sheaf $S$ the groupoid $Y(S)$ is equivalent to the category of $[\Gamma\rightrightarrows X]$-torsors over $S$, where we recall that the category of $[\Gamma\rightrightarrows X]$-torsors is given by the category of finite \'etale surjective maps $S'\to S$ equipped with a cartesian map of groupoids \[\xymatrix{S'\times_S S'\ar[r]\ar@<-.5ex>[d]\ar@<.5ex>[d] & \Gamma \ar@<-.5ex>[d]\ar@<.5ex>[d]\\
S'\ar[r] & X}\] with the obvious notion of morphism.

This description of $Y(S)$ combined with the facts that 1) the categories of finite \'etale sheaves
over $\Spec(W(R))$ and over $W(S)$ are equivalent and 2) the maps
$$
X(W(R)) \longisomap \lim_{m, n} X(W_n(R_m))\quad\text{and}\quad 
\Gamma(W(R)) \longisomap \lim_{m,n} \Gamma(W_n(R_m))
$$
are bijective (remark 4.3 of \cite{Bhatt:algebraization}) gives the claim.
\end{proof}

\begin{coro}\label{coro:algebraic-delta} 
If $S$ is $p$-torsion free, then there is a unique $\delta_{\Spec(W(R))}$-structure on the algebraisation 
$\wt{A}^{\mathrm{alg}}/\Spec(W(R))$ extending that on $\wt{A}/W(S)$.
\end{coro}
\begin{proof}
Since $S$ is $p$-torsion free, so are $\Spec(W(R))$ and $\wt{A}^{\mathrm{alg}}$.
It then follows from \ref{prop:delta-strucures-frobenius-lift} that $\wt{A}^{\mathrm{alg}}$ has a unique 
$\delta_{\Spec(W(R))}$-structure inducing compatible with its Frobenius lift.
\end{proof}

\begin{theo} 
If $S=\Spec(R)$, where $p$ is nilpotent in $R$, then there is a unique
$\delta_{\Spec(W(R))}$-structure on the algebraisation on $\wt{A}^{\mathrm{alg}}/\Spec(W(R))$ extending that on
$\wt{A}/W(S)$.
\end{theo}
\begin{proof} 
Since the moduli space of ordinary polarised abelian schemes is smooth and locally finitely presented, it
follows that there exists a $p$-torsion free formal affine scheme $S'=\colim_n \Spec(B_n)$, an ordinary
polarised abelian scheme $A'/S'$ and a morphism $S\to S'$ such that $A\cong A'\times_{S'}S$. (See the proof of
\ref{thm:can-lift}.) As the algebraisation is functorial it follows that the algebraisation of
$\wt{A}^{\mathrm{alg}}/\Spec(W(R))$ also admits a unique $\delta_{\Spec(W(R))}$-structure.
\end{proof}

\begin{rema} The theorem stated in the introduction follows from that above. Indeed, $A/R$ is an ordinary abelian scheme and if $A'/W(R)$ is any abelian scheme lifting $A/R$ and equipped with a $\delta_{W(R)}$-structure compatible with its group structure then $A'\times_{\Spec(W(R))}W(S)$ must be isomorphic to $\wt{A}/W(S)$ by \ref{theo:can-lift-p-adic} and so it follows that $A'$ must be isomorphic to $\wt{A}^\alg$ by \ref{theo:algebraisation}. Moreover, these isomorphisms are compatible with the Frobenius lifts and hence with the $\delta$-structures by \ref{lemm:hom-abelian-lambda}.
\end{rema}

\frenchspacing
\bibliography{references}
\bibliographystyle{plain}

\end{document}